\documentclass[10pt]{amsart} 

\usepackage{amsmath, amscd, amsthm}  
\usepackage{mathrsfs, bbm, amsfonts, amssymb}
\usepackage{stmaryrd}
\usepackage{graphicx}

\usepackage{youngtab}
\usepackage[all]{xy}
\usepackage{hyperref, color}
\definecolor{darkblue}{rgb}{0,0,1}
\hypersetup{pdfpagemode=UseNone, colorlinks=true, breaklinks=true, linkcolor=darkblue, menucolor=darkblue, urlcolor=darkblue, citecolor=darkblue}

\newtheorem{proposition}{Proposition}[section]

\newtheorem{lemma}[proposition]{Lemma}

\theoremstyle{remark}

\newtheorem{remark}[proposition]{Remark}
\newtheorem{example}[proposition]{Example}

\theoremstyle{definition}
\newtheorem{definition}[proposition]{Definition}

\allowdisplaybreaks

\newcommand{\RR}{\mathbbm R}

\newcommand{\CC}{\mathbbm C}

\begin{document}

\title{Generalized duality for $k$-forms}
\author{Frank Klinker}
\date{June, 1 2011}
\address{Faculty of Mathematics, TU Dortmund University, 44221 Dortmund, Germany} 
\email{\href{mailto:frank.klinker@math.tu-dortmund.de}{frank.klinker@math.tu-dortmund.de}}
\begin{abstract} 
We give the definition of a duality that is applicable to arbitrary $k$-forms. The operator that defines the duality depends on a fixed form $\Omega$. Our definition extends in a very natural way the Hodge duality of $n$-forms in $2n$ dimensional spaces and the generalized duality of two-forms. We discuss the properties of the duality in the case where $\Omega$ is invariant with respect to a subalgebra of $\mathfrak{so}(V)$. Furthermore, we give examples for the invariant case as well as for the case of discrete symmetry.
\end{abstract}
\thanks{Published in 
\href{http://dx.doi.org/10.1016/j.geomphys.2011.07.007}{
{\em J.~Geom.~Phys.} {\bf 61} (2011), no.~12, 2293-2308}}
\maketitle

\section{Introduction: Self duality and $\Omega$-duality}\label{dual}

Given a Riemannian or semi-Riemannian space $(V,g)$ of dimension $D$, the metric $g$ induces isomorphisms $*:\Lambda^kV\to\Lambda^{D-k}V$.  This so called Hodge operator has the property $*^2=\varepsilon\mathbbm{1}$ where the sign $\varepsilon$ depends on the dimension and the signature of the metric by $*^2\big|_{\Lambda^kV}=(-)^{t+k(D-k)}\mathbbm{1}$. If the dimension of $V$ is even, $D=2n$, we have a particular automorphism $*:\Lambda^nV\to\Lambda^nV$. For $\varepsilon=1$, i.e.~$D\equiv0\mod4$, we call the $n$-form $F$ self dual and anti-self dual if it is an eigenform of $*$ to the eigenvalue $1$ and $-1$, respectively, i.e.
\begin{equation}\label{selfdual}
*F=\pm F\,.
\end{equation}
Duality relations are in particular interesting for two-forms. Consider a vector bundle $E$ over the Riemannian base $(M,g)$. The curvature tensor of a connection on $E$ is a two-form on $M$ with values in the endomorphism bundle of $E$. So for $\dim M=4$ we may consider connections with (anti-)self dual curvature tensor. In the case of $E=TM$ (anti-)self duality is connected to complex structures on $M$; see \cite{AtHiSi}.  

In dimension four we may use the volume form ${\rm vol}=*1$ to rewrite (\ref{selfdual}) as
\begin{equation}
*(*{\rm vol}\wedge F)=\pm F\,.
\end{equation}
This motivates the introduction of $\Omega$-duality of two-forms in arbitrary dimension; see for example \cite{AlekCorDev1,AlekCorDev2,BauKanSin,CGK,DevDual} and \cite{Ward}. It is defined as follows. Let $\Omega$ be a four-form on $V$ and consider the symmetric operator $*_\Omega:\Lambda^2V\to\Lambda^2V$ with
\begin{equation}
*_\Omega F:=*(*\Omega\wedge F)\,.
\end{equation}
Let us suppose that $*_\Omega$ admits  real eigenvalues, then a two-form $F$ is called $(\Omega,\beta)$-dual if it obeys
\begin{equation}\label{omegadual}
*(*\Omega\wedge F)=\beta F
\end{equation}
(see \cite{AlekCorDev1}). In local coordinates with $\Omega=\Omega_{ijkl}$ and $F=F_{ij}$ the left hand side of (\ref{omegadual}) is given by $*(*\Omega\wedge F)_{ij}=\frac{1}{2}\Omega_{ijkl}F^{kl}$.

\begin{example}\label{ex1}
Consider the three-form $\theta$ in seven dimensions that is given by $\theta_{ijk}=1$ for $(ijk)=(123)$, $(435)$, $(471)$, $(516)$, $(572)$, $(624)$, $(673)$. We associate to this the four-form $\bar\theta:=*\theta$ in seven and the four-form $\Theta:= \bar\theta+\theta\wedge e_8$ in eight dimensions. The latter is self-dual, i.e. $\Theta=*_8\Theta$.\footnote{The invariant four-forms are explicitly given by $\bar\theta_{ijkl}=1$ for $(ijkl)=(1245)$, $(1276)$, $(1346)$, $(1357)$, $(2356)$, $(2437)$, $(4567)$ and $\Theta_{ijkl}=1$ for $(ijkl)=(1245)$, $(1276)$, $(1346)$, $(1357)$, $(2356)$, $(2437)$, $(4567)$, $(1238)$, $(4358)$, $(4718)$, $(5168)$, $(5728)$, $(6248)$, $(6738)$.}
The forms above are strongly related to the discussion of $\mathfrak{g}_2$ and $\mathfrak{spin}(7)$. In particular, the duality relations yield the decompositions of the adjoint representations of $\mathfrak{so}(7)$ and $\mathfrak{so}(8)$ into irreducible representations of $\mathfrak{g}_2$ and $\mathfrak{spin}(7)$, respectively:

\begin{itemize}
\item
$*_{\bar\theta}:\Lambda^2\RR^7\to\Lambda^2\RR^7$ has eigenvalues $1$ and $-2$ and the eigenspace decomposition corresponds to the decomposition of $\Lambda^2\RR^7$ with respect to $\mathfrak{g}_2$. In particular $E(1,*_{\bar\theta})=\mathbf{14}$  is the adjoint representation and $E(-2,*_{\bar\theta})=\mathbf{7}$ is the vector representation of $\mathfrak{g}_2$.
\item
$*_\Theta:\Lambda^2\RR^8\to\Lambda^2\RR^8$ has eigenvalues $1$ and $-3$. The eigenspace decomposition corresponds to the decomposition of $\Lambda^2\RR^8$ with respect to $\mathfrak{spin}(7)$. In particular $E(1,*\Theta)=\mathbf{21}$  is the adjoint representation and $E(-3,*_\Theta)=\mathbf{7}$ is the vector representation of $\mathfrak{spin}(7)$.
\end{itemize}
\end{example}

\begin{example}\label{ex2}
Consider the globally defined parallel four-form $\Omega:=\omega_1\wedge\omega_1+\omega_2\wedge\omega_2+\omega_3\wedge\omega_3$ on the quaternionic-K\"ahler manifold $(M,\omega_1,\omega_2,\omega_3)$.   

Then the operator $*_\Omega$ on $\Lambda^2TM\otimes \CC$ has eigenvalues $1$, $-\frac{1}{3}$ and $-\frac{2m+1}{3}$ corresponding to the eigenspaces $E(1,*_\Omega)=S^2Y\otimes \sigma_Z$, $E(-\frac{1}{3},*_\Omega)=\Lambda^2_0Y\otimes S^2Z$ and $E(-\frac{2m+1}{3},*_\Omega)=\sigma_Y\otimes S^2Z$. 
Here $TM\otimes\CC=Y\otimes Z$ is  the (local) decomposition of the complexified tangent bundle into a rank-$2m$ and a rank-$2$ bundle with respect to the $Sp(m)Sp(1)$-structure, and $g=\sigma_Y\otimes\sigma_Z$ the corresponding decomposition of the complexified Riemannian metric on $M$. The $1$-eigenspace is connected to half-flatness introduced by the authors in \cite{AlekCorDev1,AlekCorDev2}.
\end{example}

\begin{remark}\label{omegadualgen}
Equation (\ref{omegadual}) can be generalized further in an straight forward way. Let $\Omega$ be a $2k$-Form and $F$ be a $k$-Form. Then $*_\Omega(F):=*(*\Omega\wedge F)$ is also a $k$ form and the question whether or not $*_\Omega$ has real eigenvalues is reasonable. Such operators is discussed in \cite{Baulieu} and examples are given in \cite{DNW} for $k=4,6$ in dimension ten.
\end{remark}

\section{Duality of $k$-forms}

All examples in the previous section have in common that the $\ell$-form $\Omega$ yields a duality relation on the space $\Lambda^{\frac{\ell}{2}}V$ only. It would be preferable to give for one fixed $\Omega$ a duality relation on each $\Lambda^kV$. As we will see in Lemma \ref{lemma-restr} this is possible at least up to some mild restrictions.

\begin{definition}\label{def}
Let $\Omega\in\Lambda^\ell V$ be an $\ell$-form on $V=\RR^D$.
The {\em duality operator} $b_\Omega$ is defined by
\begin{equation}
\begin{aligned}
b_\Omega &:\Lambda^kV\to\Lambda^kV\,, \quad
F \mapsto \pi_k(\Omega\otimes F)\,.
\end{aligned}
\end{equation}
Here $\pi_k:\Lambda^\ell V\otimes\Lambda^k V\to\Lambda^kV$ denotes the projection in the decomposition of $\Lambda^\ell V\otimes \Lambda^k V$ with respect to $\mathfrak{so}(V)$; see (\ref{dec-so}).\footnote{To simplify the formulas we feel free to consider the projection up to a constant factor. This leads to the fact, that for example $b_\Omega=2*_\Omega$ for a four-form $\Omega$ acting on $\Lambda^2V$, compare Example \ref{ex1} and Lemma \ref{b-local}.}
We call the duality operator $b_\Omega$ {\sc of order $N$} if it admits $N$ distinct eigenvalues.
\end{definition}

\begin{example}
In \cite{DNW} the authors discuss two operators on three- and two-forms in dimension $D=10$. These two are covered by the special case $\ell=2k$ in Definition \ref{def}. 
\end{example}

\begin{example}
A very basic example is the following. Let $\Omega$ be a complex structure on $V=\RR^{2n}$ interpreted as two-form, i.e. $\Omega_{ij}\Omega^{j}{}_{k}=-\delta_{ik}$. Then $b_\Omega$ on $\Lambda^1V$ has eigenvalues $\pm i$ with eigenspaces $\Lambda^1_{(i)}V=\Lambda^{1,0}V$ and $\Lambda^1_{(-i)}V=\Lambda^{0,1}V$. The eigenvalues of $b_\Omega$ on $\Lambda^kV$ for $k\leq n$ are then given by $\frac{k-2q}{k}i$ for $q=0,\ldots,k$ with eigenspaces $\Lambda^k_{\left(\frac{k-2q}{k}i\right)}V=\Lambda^{k-q}(\Lambda^{1,0}V)\otimes\Lambda^q(\Lambda^{0,1}V)=\Lambda^{k-q,q}V$.
For $k>n$ we refer the reader to Remark \ref{evenremark}. For instance, the Hodge dual to $\Lambda^{n-k}_{\left(\frac{n-k-2q}{n-k}i\right)}V=\Lambda^{n-k-q,q}V$ is $\Lambda^{n+k}_{\left(\frac{n-k-2q}{n+k}i\right)}V=\Lambda^{n-q,k+q}V$.
\end{example}

\begin{remark}
The preceding example can be generalized. For $\Omega\in\Lambda^2V$  the action  of $b_\Omega$ on $\Lambda^kV$ is just the action of $\frac{1}{k}\Omega\in\mathfrak{so}(V)$ on $\Lambda^kV$. 
\end{remark}

\begin{lemma}\label{lemma-restr}
Let $\Omega\in\Lambda^\ell V$. Then $b_\Omega\neq 0$ only if $\ell$ is even and $\ell\leq 2k$.
\end{lemma}

\begin{proof}
Consider the $\mathfrak{so}(V)$-decomposition as given in (\ref{dec-so}). Then we have
\begin{align}\label{dec}
\Lambda^kV\subset\Lambda^\ell V\otimes\Lambda^kV 
& = \bigoplus_{i=0}^{\min\{k,\ell\}}\bigoplus_{j=0}^{i-1}\llbracket k+\ell-i-j,i-j\rrbracket_0 \oplus  \bigoplus_{i=0}^{\min\{k,\ell\}}\Lambda^{k+\ell-2i}V
\end{align}
if and only if $2i=\ell$ for some $i\in\big\{0,\ldots,\min\{k,\ell\}\big\}$. This is $\ell$ even and $\frac{\ell}{2}\leq\min\{k,\ell\}$ or $\ell\leq 2k$.
\end{proof}

\begin{remark}
Because of the restriction given in Lemma \ref{lemma-restr} it would be preferable, that $\ell$ is not too big. Therefore, the case $\ell=4$ is of particular interest. The main examples which have been cited so far are connected to this value. 
\end{remark}

\begin{example}\label{basicexample}
The examples from section \ref{dual}, the Hodge duality and the $\Omega$-dualities, are operators of the form $b_\Omega$. They are of order two (Hodge-duality and Example \ref{ex1})  or order three (Example \ref{ex2}) with $\ell=2k$. 
\end{example}

\begin{lemma}\label{b-local}
In local coordinates be may write $\Omega=(\Omega_{i_1\ldots i_{2m}})$. Then $b_\Omega$ and $b_\Omega^2$ are given by 
\begin{equation}\label{bone}
(b_\Omega)^{d_1\ldots d_k}{}_{i_1\ldots i_k} 
	=\Omega^{[d_1\ldots d_m}{}_{[i_1\ldots i_m }\delta^{d_{m+1}\ldots d_k]}_{i_{m+1}\ldots i_k]} 
\end{equation}
and
\begin{equation}\label{bsquare}
(b_\Omega^2)^{d_1\ldots d_k}{}_{i_1\ldots i_k} =
\delta_{j_1\ldots j_m [ i_{m+1}\ldots i_k}^{b_1\ldots b_m [d_{m+1}\ldots d_k }
\Omega^{|j_1\ldots j_m|}{}_{i_1\ldots i_m]}\Omega^{d_1\ldots d_m]}{}_{b_1\ldots b_m}
\end{equation}
respectively.
In particular $b_\Omega$ is trace free.
\end{lemma}
\begin{proof}
Consider $\Omega=(\Omega_{i_1\ldots i_{2m}})$ and $F=(F_{i_1\ldots i_k})$. Then the definition of $b_\Omega(F)$ as projection on $\Lambda^kV$ in (\ref{dec}) yields
\begin{align*}
(b_\Omega(F))_{i_1\ldots i_k} =\ &  
\Omega_{j_1\ldots j_m [ i_1\ldots i_m }F^{j_1\ldots j_m}{}_{i_{m+1}\ldots i_k]} \\
 =\ &
\delta^{a_1\ldots a_k}_{i_1\ldots i_k}\delta^{d_1\ldots d_md_{m+1}\ldots d_k}_{j_1\ldots j_m a_{m+1}\ldots a_k} \Omega^{j_1\ldots j_m}{}_{a_1\ldots a_m } F_{d_1\ldots d_k} \\
=\ &
\delta^{[d_{m+1}\ldots d_k}_{[i_{m+1}\ldots i_k} 
\Omega^{d_1\ldots d_m]}{}_{i_1\ldots i_m] } F_{d_1\ldots d_k}
\end{align*}
The square of $b_\Omega$ obeys
\begin{align*}
(b_\Omega^2(F))_{i_1\ldots i_k} 
=\ & \delta^{a_1\ldots a_k}_{i_1\ldots i_k} \Omega^{j_1\ldots j_m}{}_{a_1\ldots a_m}(b_\Omega(F))_{j_1\ldots j_m a_{m+1}\ldots a_k} \\
=\ & 
\delta^{a_1\ldots a_k}_{i_1\ldots i_k} 
\delta_{j_1\ldots j_m a_{m+1}\ldots a_k}^{b_1\ldots b_m b_{m+1}\ldots b_k}
\Omega^{j_1\ldots j_m}{}_{a_1\ldots a_m}
\Omega_{c_1\ldots c_m b_1\ldots b_m}\cdot\\
&\ \cdot
F^{c_1\ldots c_m}{}_{b_{m+1}\ldots b_k}
\\
=\ &
\delta^{a_1\ldots a_ma_{m+1}\ldots a_k}_{i_1\ldots i_mi_{m+1}\ldots i_k} 
\delta^{d_1\ldots d_md_{m+1}\ldots d_k}_{c_1\ldots c_mb_{m+1}\ldots b_k}
\delta_{j_1\ldots j_m a_{m+1}\ldots a_k}^{b_1\ldots b_m b_{m+1}\ldots b_k}\cdot\\
&\ \cdot\Omega^{j_1\ldots j_m}{}_{a_1\ldots a_m}
\Omega^{c_1\ldots c_m}{}_{ b_1\ldots b_m}
F_{d_1\ldots d_k}\\
=\ & 
\delta_{j_1\ldots j_m [ i_{m+1}\ldots i_k}^{b_1\ldots b_m  d_{m+1}\ldots d_k}
\Omega^{j_1\ldots j_m}{}_{i_1\ldots i_m]}
\Omega^{d_1\ldots d_m}{}_{ b_1\ldots b_m}
F_{d_1\ldots d_k}
\end{align*}
If we use (\ref{bone}), we see that the trace of $b_\Omega$ is given by 
\[
{\rm tr}(b_\Omega)
 =\Omega^{i_1\ldots i_m}{}_{[i_1\ldots i_m }\delta^{i_{m+1}\ldots i_k}_{i_{m+1}\ldots i_k]} 
\propto \Omega^{i_1\ldots i_m}{}_{i_1\ldots i_m}
=0\,.
\]
\end{proof}
\begin{remark}
Let $\Omega$ be an $\ell$-Form with $\ell=2m$. From (\ref{bone}) we see that the linear operator $b_\Omega$ is skew symmetric if $m$ is odd and that it is symmetric if $m$ is even. 
In particular, $b_\Omega$ is diagonalizable with purely imaginary eigenvalues if $m$ is odd and real eigenvalues if $m$ is even. 

If $b_\Omega$ is of order $N$ with different eigenvalues $\beta_1,\ldots,\beta_N$, then $b_\Omega$ solves its minimal polynomial  $\lambda^N-(\beta_1+\cdots+\beta_N) \lambda^{N-1} +\cdots +(-)^N\beta_1\cdots\beta_N=0$. 
\end{remark}

Because $b^2_\Omega$ is symmetric, it is contained in $S^2(\Lambda^kV)$. So the right hand side of (\ref{bsquare}) is an element in $\Lambda^\ell V\otimes\Lambda^\ell V$ that is embedded in $S^2(\Lambda^kV)$ via some $\delta$-tensor. If $b_\Omega$ is of order two with eigenvalues $\beta_1\neq -\beta_2$ then $b_\Omega$ has to be symmetric, too. This is enough to state the following result on duality operators of order two.

\begin{proposition}\label{1}
Let $\Omega$ be an $\ell$-form on $V$. The operator $b_\Omega$ is of order two with two eigenvalues $\beta_1\neq-\beta_2$  only if $\Lambda^\ell V\subset S^2(\Lambda^kV)$. In particular $\ell\equiv 0\mod4$.

Moreover, the projections on the two respective eigenspaces are given by
\begin{equation}
\begin{split}
\pi_{\beta_1}&=\frac{\beta_2}{\beta_2-\beta_1}\big(\mathbbm{1}-\frac{1}{\beta_2}b_\Omega\big) \\
\pi_{\beta_2}&=\frac{\beta_1}{\beta_1-\beta_2}\big(\mathbbm{1}-\frac{1}{\beta_1}b_\Omega\big) \\
\end{split}
\end{equation} 
\end{proposition}

\begin{remark}
The restriction to $\ell$ in Proposition \ref{1} is a consequence of the symmetry of $b_\Omega$ or, equivalently, of (\ref{decompositionsquare}). 
This is not a contradiction to example \ref{basicexample} where the Hodge duality operator is of degree 2 but $\ell=\dim V$ may be equal to $2\mod 4$, because in this case we have $\beta_1=-\beta_2=1$.
\end{remark}

We emphasize on the following compatibility of the duality operator with the Hodge operator.

\begin{remark}\label{evenremark}
\begin{itemize}
\item
Consider $V$ to be of dimension $D$ and let $\Omega\in\Lambda^{2m}V$ such that $b_\Omega$ is defined on $\Lambda^kV$ as well as on $\Lambda^{D-k}V$. Then we have
\begin{equation}\label{eigenvalHodge}
{\textstyle \binom{D-k}{m}} *\,b_\Omega\,*   = (-1)^{k(D-k)} {\textstyle \binom{k}{m}} b_\Omega\,
\end{equation}
where the action is on $\Lambda^kV$.

In particular, if $F\in\Lambda^kV$ is an eigenform of $b_\Omega$ to the eigenvalue $\beta$, then $*F\in\Lambda^{D-k}V$ is an eigenform of $b_\Omega$ to the eigenvalue $\beta'=\frac{\binom{k}{m}}{\binom{D-k}{m}}\beta$, i.e.\ $\Lambda^k_{(\beta)}\approx \Lambda^{n-k}_{(\beta')}$ via $*$.
\item
If we consider $V$ of dimension $4m$ and $\Omega\in\Lambda^{2m}V$ then for all $F\in\Lambda^{2m}V$ we have
\begin{equation}
*b_\Omega(F)=b_{*\Omega}(F)=b_{\Omega}(*F) \,.
\end{equation}
In the case that $\Omega$ is either self-dual or anti self-dual, i.e.~$*\Omega=\pm\Omega$ we have $* b_\Omega(F)=\pm b_{\Omega}(F)$. Therefore, for $F\in\Lambda^{2m}_{(\beta)}$ we have $*F=\pm F$ or $\beta=0$, i.e. $(\Lambda^{2m}V)^\mp \subset \Lambda^{2m}_{(0)}$. We will recall this fact in Proposition \ref{Theta-4}.
\end{itemize}
\end{remark}

\section{Invariant duality operators}
\subsection{General properties of invariant duality operators}

Let $b_\Omega:\Lambda^kV\to\Lambda^kV$ be a duality operator of order $N$ associated to $\Omega\in\Lambda^\ell V$. Consider $\Omega$ to be invariant with respect to a subalgebra $\mathfrak{h}\subset\mathfrak{so}(V)$. Then $b_\Omega$ is invariant under $\mathfrak{h}$ as well. If $\Lambda^kV= W_1\oplus\ldots\oplus W_r$ is the decomposition into irreducible representation spaces with respect to $\mathfrak{h}$, then 
\[
b_\Omega\big|_{W_\alpha}=\beta_\alpha\mathbbm{1}
\]
for some number $\beta_\alpha$ due to Schur's Lemma, i.e.~$W_\alpha\subset \Lambda^k_{(\beta_\alpha)}$. In particular, $b_\Omega$ is diagonalizable with (not necessarily distinct) eigenvalues $\beta_1,\ldots,\beta_r$. Because $b_\Omega$ is trace free, we have in this special situation
\begin{equation}\label{trace}
\sum_{\alpha=1}^r \beta_\alpha {\rm dim}(W_\alpha) = 0\,.
\end{equation}

\begin{definition}\label{defperf}
Let $\Omega\in\Lambda^{\ell}V$ be invariant under a subalgebra $\mathfrak{h}\subset\mathfrak{so}(V)$. Then $b_\Omega:\Lambda^kV\to\Lambda^kV$ is called {\sc perfect} if it is of order $r$ where $r$ is the number of irreducible submodules of $\Lambda^kV$. 
\end{definition}

If $\Omega\in\Lambda^\ell V$ is invariant with respect to a subalgebra $\mathfrak{h}\subset\mathfrak{so}(V)$, then this is the same as to say that it spans a singlet within the decomposition of the $\mathfrak{so}(V)$-representation $\Lambda^\ell V$ into irreducible $\mathfrak{h}$-representations.

As noticed before the case $\ell=4$ is of particular interest. On the one hand due to the $\Omega$-duality of two-forms as in (\ref{omegadual}), on the other hand due to the restriction cf.~Lemma \ref{lemma-restr}. 
An $\mathfrak{h}$-invariant four-form may be constructed via an $\mathfrak{h}$-invariant metric as the $\Lambda^4V$-part of $S^2\mathfrak{h}\subset S^2(\Lambda^2V)$. 
This is in particular possible in the cases where $\mathfrak{h}$ is a holonomy algebra; see \cite{AlekCorDev2}. The four-forms from the examples in section \ref{dual}, that deal with $\mathfrak{spin}(7)$, $\mathfrak{g}_2$, and $\mathfrak{sp}(n)\oplus\mathfrak{sp}(1)$, are of this type.
How they occur as a singlet and that they are unique up to a multiple can be seen as follows.

For instance, the four-form $\bar\theta\in\Lambda^4\RR^7$, or it Hodge-dual $\theta\in\Lambda^3\RR^7$, considered in Example \ref{ex1} is the singlet in the $\mathfrak{g}_2$-decomposition 
$\Lambda^4\RR^8\simeq\Lambda^3\RR^7 = \mathbf{27}\oplus\mathbf{7}\oplus\mathbf{1}$.
The same is true for the four-form $\Theta\in\Lambda^4\RR^8$ also from Example \ref{ex1}. It represents the singlet in the $\mathfrak{spin}(7)$-decomposition $\Lambda^4\RR^8=\mathbf{35}\oplus\mathbf{27}\oplus\mathbf{7}\oplus\mathbf{1}$.
Moreover the four-form from example \ref{ex2} represents the singlet in the $\mathfrak{sp}(n)\oplus\mathfrak{sp}(1)$-decomposition of  $\Lambda^4V$ for $V=\CC^{4n}$. Let us recall the splitting $V=E\otimes H$ with $H=\CC^2$ and $H=\CC^{2n}$, then it can be located in the following way. 
The splitting yields $\Lambda^2V=\Lambda^2(H\otimes E)= (\mathbf{1}\otimes S^2E)\oplus (S^2H\otimes \Lambda_*^2H)\oplus (S^2E\otimes \mathbf{1})$. Then the singlet in $\Lambda^4V$ coincides with singlet in $\mathbf{1}\otimes\mathbf{1}\subset (\mathbf{1}\otimes S^2E)\otimes (\mathbf{1}\otimes S^2E)\subset \Lambda^2V\otimes\Lambda^2V$ coming from the trace in $S^2E\otimes S^2E$. In particular, these three examples yield perfect duality operators on the space of two forms.

A list and the explicit construction of invariant four-forms in dimension $D\leq 8$ for subgroups of $\mathfrak{so}(D)$ is given in \cite{CDF}. In particular, the authors give a four-form depending on three real parameters, that yield the decomposition of $\Lambda^2\RR^8$ for  $\mathfrak{h}=\mathfrak{u}(4)=\mathfrak{su}(4)\oplus\mathfrak{u}(1)$, and $\mathfrak{spin}(7)$, as well as $\mathfrak{su}(4)$. The decomposition for $\mathfrak{u}(4)$ is not perfect, whereas the remaining two are.

The authors in \cite{DNW} discuss the $\Omega$-duality in dimension $D=10$ in the generalized sense cf.~Remark \ref{omegadualgen}. They construct a six-form and its associated Hodge-dual four-form invariant under
$\mathfrak{su}(4)\oplus\mathfrak{u}(1)\subset\mathfrak{so}(8)\oplus\mathfrak{u}(1)\subset\mathfrak{so}(10)$. The corresponding eigenspace decompositions of $\Lambda^3\RR^{10}$ and $\Lambda^2\RR^{10}$ are not perfect in the sense of Definition \ref{defperf}.

\subsection{The $\mathfrak{spin}(7)$- and $\mathfrak{g}_2$-duality}

The first two examples of this section, i.e.~Pro\-positions \ref{Theta-3} and \ref{Theta-4} make use of the $\mathfrak{spin}(7)$-invariant four-form to give the eigen\-space decomposition of $\Lambda^3\RR^8$ and $\Lambda^4\RR^8$. In particular, the duality-operator is perfect in both cases and therefore, the eigenspace decomposition coincides with the decomposition into irreducible representations. This extends the result from Example \ref{ex1} to all forms on $\RR^8$. 
Of course, these $\mathfrak{spin}(7)$-decompositions in terms of the invariant tensor $\Theta$ are not new, but very common in the literature, see e.g.~\cite{Gauntlett2003,Tsimpis2006,Fernandez} or \cite{bryant}.\footnote{For his description the author in \cite{Fernandez} uses the concept of vector cross products, of which a nice discussion and classification is given in \cite{vector}.} Nevertheless, they yield nice examples how the known results fit in our duality framework.

\begin{proposition}\label{Theta-3}
Consider the   $\mathfrak{spin}(7)$-invariant four-form $\Theta$ on $V=\RR^8$. Then $b_\Theta:\Lambda^3V\to\Lambda^3V$ with $(b_\Theta)_{lmn}{}^{ijk} = \Theta_{[lm}{}^{[ij}\delta_{n]}^{k]}$  is a perfect duality operator of order two that obeys 
\[
(b_\Omega)^2=\frac{8}{3}{\rm id}-\frac{10}{3}b_\Omega.
\]
The eigenvalues are $-4$ and $\frac{2}{3}$ corresponding to the eight-dimensional and $48$-dimensional $\mathfrak{spin}(7)$-invariant subspaces of $\Lambda^3V$.\footnote{${\mathbf 8}$ is the spin representation and the $\mathbf{48}$ is the spin-$\frac{3}{2}$ representation of $\mathfrak{so}(7)$. The latter is given by  vector-spinors which obey $\gamma^\mu\psi_\mu=0$.}
\end{proposition}

\begin{proof}
The traces of the eight-tensor $\widetilde{\Theta}=\Theta\otimes\Theta$ have components in the skew-symmetric parts of $S^2(\Lambda^4\RR^8))$ only. They are explicitly given by
\begin{equation}\label{traces}
\begin{aligned}
\Theta_{ijko}\Theta^{lmno}&=6\delta^{lmn}_{ijk}-9\Theta_{[ij}{}^{[lm}\,\delta_{k]}^{n]}\,, \\
\Theta_{ijmn}\Theta^{klmn}&=12\delta^{kl}_{ij}-4\Theta_{ij}{}^{kl}\,, \\
\Theta_{iklm}\Theta^{jklm}&=42\delta^j_i\,, \\
\Theta_{ijkl}\Theta^{ijkl}&=336\,.
\end{aligned}
\end{equation}
This gives
\begin{align*}
b^2_\Theta(F)_{i_1i_2i_3}
=\ 	&\delta^{b_1b_2d_3}_{j_1j_2[i_3}\,\Theta^{j_1j_2}{}_{i_1i_2]}\,\Theta^{d_1d_2}{}_{b_1b_2}F_{d_1d_2d_3} \\
=\ 	& \tfrac{1}{3}\big(\delta^{b_1b_2}_{j_1j_2}\delta^{d_3}_{[i_3}\,\Theta^{j_1j_2}{}_{i_1i_2]}\,
		\Theta^{d_1d_2}{}_{b_1b_2}F_{d_1d_2d_3} \\
	&+2 \delta^{d_3b_1}_{j_1j_2}\delta^{b_2}_{[i_3}\,\Theta^{j_1j_2}{}_{i_1i_2]}\,
		\Theta^{d_1d_2}{}_{b_1b_2}F_{d_1d_2d_3}\big) \\
=\ 	& \tfrac{1}{3}\big( \Theta_{d_1d_2j_1j_2}\,\Theta^{j_1j_2}{}_{[i_1i_2}F_{i_3]d_1d_2} 
	 -2 \Theta_{j_2}{}^{d_1d_2 m}\,\Theta_{j_2j_1[i_1i_2}g_{i_3]m}F_{d_1d_2}{}^{j_1}\big) \\
=\ 	& \tfrac{1}{3}\big( 12\delta^{d_1d_2}_{[i_1i_2}-4\Theta^{d_1d_2}{}_{[i_1i_2}\big)F_{i_3]}{}^{d_1d_2} \\
	&  -\tfrac{2}{3}\big(6\delta^{d_1d_2m}_{j_1[i_1i_2}-9\Theta^{[d_1d_2}{}_{b_1b_2}\,
		\delta_{b_3}^{m]}\delta^{b_1b_2b_3}_{j_1[i_1i_2}\big) g_{i_3]m}F_{d_1d_2}{}^{j_1} \\
=\ 	&4 F_{i_1i_2i_3}-\tfrac{4}{3} (b_\Theta(F))_{i_1i_2i_3} 
		-4\big( \tfrac{1}{3}\delta^m_{j_1}\delta^{d_1d_2}_{[i_1i_2}g_{i_3]m} F_{d_1d_2}{}^{j_1}\\ 
    &+\tfrac{2}{3}\delta^{d_1}_{j_1}\delta^{d_2m}_{[i_1i_2}g_{i_3]m} F_{d_1d_2}{}^{j_1} \big) \\
	&+6 \big( \tfrac{1}{3} \Theta^{d_1d_2}{}_{b_1b_2}\,\delta^m_{b_3}
	 +\tfrac{2}{3} \Theta^{m d_1}{}_{b_1b_2}\,\delta_{b_3}^{d_2} \big)\delta^{b_1b_2b_3}_{j_1[i_1i_2} g_{i_3]m}F_{d_1d_2}{}^{j_1} \\ 
=\ 	& 4 F_{i_1i_2i_3}-\tfrac{4}{3} (b_\Theta(F))_{i_1i_2i_3} -\tfrac{4}{3}F_{i_1i_2i_3}\\
	&+2 \big( \tfrac{2}{3}\Theta^{d_1d_2}{}_{j_1[i_1}\,\delta^m_{i_2} 
	 +\tfrac{1}{3}\delta^m_{j_1}\,\Theta^{d_1d_2}{}_{i_1i_2}\big)g_{i_3]m}F_{d_1d_2}{}^{j_1} \\
	&+4\big(\tfrac{2}{3}\Theta^{m d_1}{}_{j_1[i_1}\,\delta_{i_2}^{d_2} 
	 +\tfrac{1}{3}\delta_{j_1}^{d_2} \,\Theta^{m d_1}{}_{i_1 i_2}\big) g_{i_3]m}F_{d_1d_2}{}^{j_1} \\
=\ & \tfrac{8}{3} F_{i_1i_2i_3}
	-\tfrac{4}{3} (b_\Theta(F))_{i_1i_2i_3} 
	+\tfrac{2}{3} \Theta_{d_1d_2[i_1i_2} F^{d_1d_2}{}_{i_3]} 
	-\tfrac{8}{3}\Theta^{d_1d_2}{}_{[i_1i_2}  F_{i_3]d_1d_2} \\
=\ 	&\tfrac{8}{3} F_{i_1i_2i_3}-\tfrac{10}{3} (b_\Theta(F))_{i_1i_2i_3}\,.
\end{align*}
The eigenvalues of $b_\Theta$ are the zeros of $\beta^2+\frac{10}{3}\beta-\frac{8}{3}$ that are 
$\frac{2}{3}$ and $-4$. The eigenspaces are given by  $\Lambda^3_{(\frac{2}{3})} V=\mathbf{48}$ and $\Lambda^3_{(-4)}V=\mathbf{8}$ due to $(-4).8+\frac{2}{3}.48=0$. 
\end{proof} 

\begin{lemma}\label{example-Theta}
Let $\Theta$ and $V$ as before and consider the duality map $b_\Theta:\Lambda^4V\to\Lambda^4V$ given by
$
b_\Theta(F)_{ijkl}=\Theta^{mn}{}_{[ij}F_{kl]mn}\,.
$
This operator obeys\footnote{We postpone the calculations to Appendix \ref{appencalc}.}
\begin{align}
b_\Theta^2(F)_{ijkl} =\ & 
     	\tfrac{1}{6}\Theta^{mn}{}_{[ij}\Theta^{op}{}_{kl]}F_{mnop} 
		+\tfrac{2}{3}F_{ijkl} -\tfrac{8}{3}b_\Theta(F)_{ijkl} \,,
	\label{b2} \\
b_\Theta^3(F)_{ijkl} =\ &
		\tfrac{4}{3}F_{ijkl} + \tfrac{2}{3}b_\Theta(F)_{ijkl} -\tfrac{10}{3}b_\Theta^2(F)_{ijkl}
		+\tfrac{2}{9}\Theta_{[ijk}{}^{p}\Theta^{rsn}{}_{l]}F_{pnrs} \,,
 	\label{b3}\\
b_\Theta^4(F)_{ijkl} =\ & 
		4 b_\Theta(F)_{ijkl} - \tfrac{8}{3}b^2_\Theta(F)_{ijkl} 
		-\tfrac{13}{3}b^3_\Theta(F)_{ijkl}+\tfrac{1}{9}\Theta_{ijkl}\Theta^{prsn}F_{prsn} \,,
	\label{b4}\\
b_\Theta^5(F)_{ijkl}=\ & 
		-\tfrac{25}{3}b_\Theta^4(F)_{ijkl}-20b_\Theta^3(F)_{ijkl}
		-\tfrac{20}{3}b_\Theta^2(F)_{ijkl}+16b_\Theta(F)_{ijkl} \,.
	\label{b5}
\end{align}
\end{lemma}

\begin{remark}\label{remval}
Equation (\ref{b5}) yields, that $b_\Theta$ is a null of the polynomial
\begin{equation}
\beta^5+\tfrac{25}{3}\beta^4 +20\beta^3+\tfrac{20}{3}\beta^2-16\beta
   =\beta(\beta+4)(\beta+3)(\beta+2)(\beta-\tfrac{2}{3})\,.
\end{equation}
so that the possible eigenvalues are $\beta=0,-2,-3,-4$, and $\frac{2}{3}$. 
\end{remark}

\begin{proposition}\label{Theta-4}
Let $\Theta$ and $V$ as before. The duality operator $b_\Theta:\Lambda^4V\to\Lambda^4V$ with $(b_\Theta)_{i_1i_2i_3i_4}{}^{j_1j_2j_3j_4}=\Theta_{[i_1i_2}{}^{[j_1j_2}\delta^{j_3j_4]}_{i_3i_4]}$ is a perfect duality operator of order four.
The eigenspaces of $b_\Theta$ and the irreducible representations of $\Lambda^4V$ with respect to $\mathfrak{spin}(7)$ correspond in the following way:
\begin{equation}\label{eigen}
\Lambda^4_{(0)}V		=  	{\mathbf{35}}\,,
\quad
\Lambda^4_{(-4)}V	= 	{\mathbf{1}}\,,
\quad
\Lambda^4_{(-2)}V	=	{\mathbf{7}}\,,
\quad
\Lambda^4_{(\frac{2}{3})}V = {\mathbf{27}}\,.
\end{equation}
The minimal polynomial is consequently given by 
\begin{equation}\label{minimal}
\beta(\beta+2)(\beta+4)(\beta-\tfrac{2}{3})=\beta^4+\tfrac{16}{3}\beta^3+4\beta^2-\tfrac{16}{3}\beta\,.
\end{equation}
\end{proposition}

We know that $\Lambda^4V$ decomposes into four irreducible representations of dimension $1$, $7$, $27$ and $35$ with respect to $\mathfrak{spin}(7)$. Therefore, one of the values from Remark \ref{remval} is not an eigenvalue. In principle, we do not need this information to sort one of the values out. Nevertheless, the following proof of Proposition \ref{Theta-4} will implicitly make use of it. 

\begin{proof}

First we show that $-4$, $0$ and $-2$ occur as eigenvalues and that the spaces of the right hand sides of (\ref{eigen}) are subsets of the respective eigenspaces. 

In particular, at least on part of the zero-eigenspace is given by $\mathbf{35}=(\Lambda^{4}V)^-\subset\Lambda^4_{(0)}V$ due to the self-duality of $\Theta$ and Remark \ref{evenremark}.

From (\ref{traces}) we immediately get $b_\Theta(\Theta)=-4\Theta$, such that $\mathbf{1}=\RR\Theta\subseteq \Lambda^{4}_{(-4)}V$.

The next element we insert into $b_\Theta$ is $F_{ijkl}=\alpha_{m[i}\Theta^m{}_{jkl]}$ for $\alpha\in\Lambda^2V$: 
\begin{align*}
b_\Theta(F)_{ijkl}
=\ & \Theta^{mn}{}_{[ij}\delta^{abcd}_{kl]mn}\alpha_{oa}\Theta^o{}_{bcd}\\
=\ &\tfrac{1}{2}\Theta^{mn}{}_{[ij}\big(\delta^{a}_{k}\delta^{bcd}_{l]mn}
	+\delta^{bcd}_{kl]n}\delta^{a}_{m}\big)\alpha_{oa}\Theta^o{}_{bcd} \\
=\ &\tfrac{1}{2}\Theta^{mn}{}_{[ij}\alpha_{k}{}^o\Theta_{l]omn}
	+\tfrac{1}{2}\Theta^{mn}{}_{[ij}\Theta^o{}_{kl]n}\alpha_{om}\\
=\ &\tfrac{1}{2}\delta^{abcd}_{ijkl}g_{dd'}\Theta^{mn}{}_{ab}\alpha_{co}\Theta_{mn}{}^{d'o}
    +\tfrac{1}{2}\delta^{abcd}_{ijkl}g_{aa'}g_{bb'}\Theta^{ma'b'n}\Theta_{ocdn}\alpha^{o}{}_{m}\\
=\ &\tfrac{1}{2}\delta^{abcd}_{ijkl}g_{dd'}\big(12\delta_{ab}^{d'o} -4\Theta_{ab}{}^{d'o}\big)\alpha_{co}
    +\tfrac{1}{2}\delta^{abcd}_{ijkl}g_{aa'}g_{bb'}\big(6\delta^{ma'b'}_{ocd} \\
   &-9\Theta_{[oc}{}^{[ma'}\delta_{d]}^{b']}\big) \alpha^{o}{}_{m}\\
=\ &-2 \alpha_{o[i}\Theta^{o}{}_{jkl]}   -\tfrac{1}{2}\delta^{abcd}_{ijkl}g_{aa'}g_{bb'}
	\alpha^{o}{}_{m}\big(4\Theta_{oc}{}^{ma'}\delta_{d}^{b'} +2\Theta_{oc}{}^{a'b'}\delta_{d}^{m} \\
   &	+2\Theta_{cd}{}^{ma'}\delta_{o}^{b'}+\Theta_{cd}{}^{a'b'}\delta_{o}^{m}\big)\\
=\ &-2 \alpha_{o[i}\Theta^{o}{}_{jkl]} -\delta^{abcd}_{ijkl}\big(-\alpha_{o[i}\Theta^{o}{}_{jkl]} 
		+\alpha_{m[i}\Theta^{m}{}_{jkl]}\big) \\
=\ &-2 \alpha_{o[i}\Theta^{o}{}_{jkl]}   
\end{align*}
therefore\footnote{We recall the decomposition of $\Lambda^2V$ as given in Example \ref{ex1} and that we have to double the eigenvalues given there, when we consider $b_\Theta$. In particular, $\alpha_{m[i}\Theta^m{}_{jkl]}=0$ for $\alpha\in\Lambda^2_{(2)}V$.} $\mathbf{7}=\big\{\alpha_{m[i}\Theta^m{}_{jkl]}\,;\,\alpha\in\Lambda^2_{(-6)}V \big\}\subset\Lambda^4_{(-2)}V$.

There is a space of dimension $27$ left, which cannot be decomposed further without getting more singlets in $\Lambda^4V$. Therefore it is irreducible, and has to be a subspace of one of the eigenspaces. The trace formula 
$0\cdot 35+ (-4)\cdot 1+(-2)\cdot 7+\beta \cdot 27 = 0$ is only solved by $\beta=\frac{2}{3}$. Such that equality in (\ref{eigen}) follows.

The above calculations and (\ref{b2})-(\ref{b4}) yield  the following decomposition of $\widetilde{\Theta}=\Theta\otimes\Theta$
\begin{equation}
\Theta_{ijkl}\Theta^{mnop} =  
 			- 42\Theta^{[mn}{}_{[ij}\delta^{op]}_{kl]} 
			+2\Theta_{[ijk}{}^{[m}\Theta_{l]}{}^{nop]}
			+3\Theta_{[ij}{}^{[mn}\Theta_{kl]}{}^{op]}\,.
\end{equation}
In contrast to its traces, the full eight-tensor $\widetilde{\Theta}$  has contributions not only from the skew-symmetric parts $\Lambda^8V$, $\Lambda^4V$, and $\Lambda^0V$ but also from $\llbracket 6,2 \rrbracket_0$ and $\llbracket 4,4 \rrbracket_0$. 
\end{proof}

\begin{remark}
We complete the discussion of the invariant $\mathfrak{spin}(7)$-four-form by adding the missing result for the  closely related invariant  $\mathfrak{g}_2$-four-form, $\bar\theta$; see Example \ref{ex1}. 

The minimal polynomial of $b_{\hat\theta}:\Lambda^3\RR^7\to\Lambda^3\RR^7$ is $\beta^3+\frac{16}{3}\beta^2+4\beta-\frac{16}{3}$ and the eigenspaces are $\Lambda^3_{(-4)}\RR^7=\mathbf{1}$, $\Lambda^3_{(-2)}\RR^7=\mathbf{7}$, and $\Lambda^3_{(\frac{2}{3})}\RR^7=\mathbf{27}$.

\end{remark}

\subsection{Lifting to higher dimensions}

There are two straightforward ways to lift an $\ell$-form $\Omega$ on $\RR^n$ to $\RR^{D}$ for $D>n$. First we consider the trivial lift given by an $\ell$-form that lives only on the $n$-space perpendicular to a specified $(D-n)$-plane. We denote this first lift by the same Symbol $\Omega$. Secondly, we consider the $*_D$-dual to this first lift, i.e.~ the $(D-\ell)$-form $\hat\Omega=*_{D}\Omega$. If $\Omega$ is $\mathfrak{g}$-invariant, these lifts are invariant with respect to $\mathfrak{g}\oplus\mathfrak{so}(D-n)$. We will discuss these two constructions for the $\mathfrak{spin}(7)$-invariant four-form in dimension eight from the preceding section and its lifts to dimension ten. The maximal invariant subalgebra is $\mathfrak{spin}(7)\oplus\mathfrak{so}(2)=\mathfrak{spin}(7)\oplus\mathfrak{u}(1)$.

We specify the $e_9\wedge e_{10}$-plane and we consider $\Theta$ to live on ${\rm span}\{e_i\}_{i\leq 8}=\RR^8$.
With respect to the decomposition $\RR^{10}=\RR^8\oplus\RR^2$ the $k$-forms split as
\begin{equation}\label{decLambda}
\Lambda^k\RR^{10} 
=\Lambda^k\RR^8 \oplus \Lambda^{k-1}\RR^8\otimes \RR^2 \oplus \Lambda^{k-2}\RR^8\otimes\Lambda^2\RR^2\,.
\end{equation}
The trivial lift of $\Theta$ now yields for $k\geq 3$ a duality operator which is given by $b_{\Theta}=b_\Theta\otimes\mathbbm{1}$ on each summand. The eigenspace decomposition for $k=3,4$ can immediately be read from the preceding sections. 
 Moreover, in the case $k=5$ we can furthermore use the symmetry $*_{10}(\Lambda^5\RR^8)=\Lambda^3\RR^8\otimes\Lambda^2\RR^2$ such that the missing decomposition follows from $b_\Theta$ on $\Lambda^3\RR^8$ alone, and the eigenvalues and eigenspaces correspond as in (\ref{eigenvalHodge}) from Remark \ref{evenremark}.
In particular, the duality operator is not perfect in all cases, due to the doubling from the second summand in the right hand side of (\ref{decLambda}).

Secondly we consider the six-form $*_{10}\Theta$. Because $\Theta$ lives on $\RR^8\subset\RR^{10}$ we  have $*_{10}\Theta = *_8\Theta\wedge\epsilon =\Theta\wedge \epsilon$ which we will denote by $\hat\Theta$. Here $\epsilon$ denotes the volume-form on $\RR^2\subset\RR^{10}$. Although this six-form is directly connected to the one before, we get a different behavior of the eigenspaces. In fact, it turns out, that the restriction of $b_{\hat\Theta}$ to $\Lambda^k\RR^{10}\big/\,{\rm ker}(b_{\hat\Theta})$ is perfect for $k=3,4$. For $k=5$ the operator is not perfect, but the two basic $\mathfrak{spin}(7)$-representations of dimension seven and eight correspond to the same non-vanishing eigenvalue.

We will state the results for $k=5,4,3$ and again postpone the calculations for the case $k=5$ to the appendix. That hopefully will convince the reader that the calculations for the remaining cases can be performed similarly.

For the case $k=5$ we need the following lemma.
\begin{lemma}\label{lemmak=5}
We consider the maps
\begin{equation}\label{formk=5}\begin{aligned}
d_\Theta & :\Lambda^5\RR^8\to\Lambda^3\RR^8,\quad d_\Theta(F)_{lmn}=\Theta_{ijk[l}F^{ijk}{}_{mn}\,,
	\\
\tilde d_\Theta & : \Lambda^3\RR^8\to\Lambda^5\RR^8,\quad \tilde d_\Theta(F)_{jklmn}=\Theta_{i[jkl}F^i{}_{mn]}\,.
\end{aligned}\end{equation}
These maps are isomorphisms and connected to $b_\Theta$ and to the Hodge operator via
\begin{equation}\label{calck=5}
d_\Theta\circ\tilde d_\Theta  	= -\frac{6}{5}id+\frac{3}{2}b_\Theta(F)\,,
\quad\text{and}\quad
*\, d_\Theta\, * 				= - 20\tilde d_\Theta \,.
\end{equation}
A consequence of this is $\tilde d_\Theta \circ d_\Theta =- *\, d_\Theta\circ \tilde d_\Theta , *$.
\end{lemma}
\begin{proof}
The identities in (\ref{calck=5}) are verified in the appendix. Due to Schur's Lemma $d_\Theta$ and $\tilde d_\Theta$ are proportional to the identity when restricted to the eigenspaces of $b_\Theta$ and moreover they are non-vanishing due to (\ref{calck=5}).
\end{proof}
If we use lemma \ref{lemmak=5} and the calculations from the appendix we get the next result.
\begin{proposition}\label{propositionk=5}
Let $\hat\Theta$ be the lift of $\Theta$ to ten dimension given by $\hat\Theta=\Theta\wedge\epsilon$. 
If we consider the decomposition of $\Lambda^5\RR^{10}$ given by (\ref{decLambda}),
then $b_{\hat\Theta}:\Lambda^5\RR^{10}\to\Lambda^5\RR^{10}$ is given by 
\begin{equation}
b_{\hat\Theta} =
\begin{pmatrix} 
& & 6\tilde d_\Theta\otimes * \\ 
& \frac{9}{5} b_\Theta \otimes * & \\
\tfrac{3}{10} d_\Theta\otimes * & & \end{pmatrix}\,.
\end{equation}
If we denote the $\pm i$-eigenspaces of $*_2$ on $\RR^2$ by $\RR_\pm$, the eigenvalues and eigenspaces of $b_{\hat\Theta}$ and their dimensions are given by
\begin{equation*}
{\renewcommand{\arraystretch}{2}\begin{array}{|c|c|c|}\hline
0
	& \Lambda^{4}_{(0)}\RR^8\otimes \RR_+ \oplus \Lambda^{4}_{(0)}\RR^8\otimes\RR_-
	& 35+35=70\\\hline
\pm \frac{32}{5}i
	& \Lambda^4_{(-4)}\RR^8\otimes \RR_\mp
	& 2\times 1 \\\hline
\pm \frac{18}{5}i
	& \left\{\left(\mp \frac{5}{3}i\,\tilde d_\Theta(F),F\wedge \epsilon\right) \big| F\in\Lambda^3_{(-4)}\RR^8\right\} \oplus \Lambda^4_{(-2)}\RR^8\otimes \RR_\mp
	& 2\times (8+7)\\\hline
\pm \frac{6}{5}i
	& \Lambda^4_{(\frac{2}{3})}\RR^8\otimes \RR_\pm 
	& 2\times 27 \\\hline
\pm\frac{3}{5}i
	& \left\{\left(\mp 10 i\,\tilde d_\Theta(F),F\wedge \epsilon\right)\big| F\in\Lambda^3_{(\frac{2}{3})}\RR^8\right\}
	& 2\times 48\\\hline
\end{array}}
\end{equation*}
\end{proposition}
The first summand in the third row and the space in the last row are subspaces of 
$\Lambda^5_{(-\frac{6}{5})}\RR^8\oplus\Lambda^3_{(-4)}\RR^8\otimes\epsilon$ and 
$\Lambda^5_{(\frac{1}{5})}\RR^8\oplus\Lambda^3_{(\frac{2}{3})}\RR^8\otimes\epsilon$, respectively.

Similar to Lemma \ref{lemmak=5} we get the following.
\begin{lemma}\label{lemmak=4}
Conisider the maps 
\begin{equation}\begin{aligned}
c_\Theta: & \Lambda^4\RR^8\to\Lambda^2\RR^8,\quad c_\Theta(F)_{ij}=\Theta_{klm[i}F^{klm}{}_{j]}\,,
	\\
\tilde c_\Theta: & \Lambda^2\RR^8\to\Lambda^4\RR^8,\quad \tilde c_\Theta(F)_{ijkl}=\Theta_{m[ijk}F^{m}{}_{l]}\,.
\end{aligned}\end{equation}
Their kernels are 
${\rm ker}(c_\Theta)=\Lambda^4_{(0)}\RR^8\oplus\Lambda^4_{(-4)}\RR^8\oplus\Lambda^4_{(\frac{2}{3})}\RR^8$ and ${\rm ker}(\tilde c_\Theta)=\Lambda^2_{(2)}\RR^8$ and the restrictions to $\Lambda^4_{(-2)}\RR^8$ and $\Lambda^2_{(-6)}\RR^8$ obey
\begin{equation}\label{iden}
c_\Theta\circ \tilde c_\Theta\big|_{\Lambda^2_{(-6)}\RR^8}=-24{id}\quad\text{and}\quad
\tilde c_\Theta\circ c_\Theta\big|_{\Lambda^4_{(-2)}\RR^8}=-24{id}\,.
\end{equation}
\end{lemma}
\begin{proof}
The statements follow from calculations similar to those for the case $k=5$ and from Schur's Lemma together with the decompositions in Example \ref{ex1} and Proposition \ref{Theta-4}.
\end{proof}
From Lemma \ref{lemmak=4} we get a result similar to the previous Proposition.
\begin{proposition}
Let $\hat\Theta$ be the lift of $\Theta$ as before and consider the decomposition of $\Lambda^4\RR^{10}$ given by (\ref{decLambda}).
Then $b_{\hat\Theta}:\Lambda^4\RR^{10}\to\Lambda^4\RR^{10}$ is given by 
\begin{equation}
b_{\hat\Theta} =
\begin{pmatrix} 
& & 6\tilde c_\Theta\otimes * \\ 
& \tfrac{9}{4}b_\Theta \otimes * & \\
\tfrac{1}{2} c_\Theta\otimes * & & \end{pmatrix}\,.
\end{equation}
the eigenvalues and eigenspaces of $b_{\hat\Theta}$ as well as their dimensions are given by
\begin{equation*}
{\renewcommand{\arraystretch}{2}\begin{array}{|c|c|c|}\hline
0 		
	& \Lambda^4_{(0)}\RR^8\oplus\Lambda^4_{(-4)}\RR^8
			\oplus\Lambda^4_{(\frac{2}{3})}\RR^8 \oplus\Lambda^2_{(2)}\RR^8\otimes\epsilon 
	& 35+1+27+21=84 \\\hline
\pm 9i	
	& \Lambda^3_{(-4)}\RR^8\otimes \RR_\mp	
	& 2\times 8 \\\hline
\pm \frac{3}{2}i	
	& \Lambda^3_{(\frac{2}{3})}\RR^8\otimes \RR_\pm 
	& 2\times 48 \\\hline
\pm 6\sqrt{2}i		
	& \left\{ \left(\mp\tfrac{i}{\sqrt{2}}\tilde c_\Theta(F), F \wedge \epsilon\right) \big| 
			F\in\Lambda^2_{(-6)}\RR^8\right\}	
	& 2\times 7\\\hline
\end{array}}
\end{equation*}
\end{proposition}
The space in the last row is a subspace of $\Lambda^4_{(-2)}\RR^8\oplus\Lambda^2_{(-6)}\RR^8\otimes\epsilon$ to and can also be written as 
$\left\{ \left(F, \mp\tfrac{i}{12\sqrt{2}} c_\Theta(F)\wedge \epsilon\right) \big| F\in\Lambda^4_{(-2)}\RR^8\right\}$ 
due to (\ref{iden}). 

To complete the discussion we add the result for $k=3$.
\begin{proposition}
With $\hat\Theta$ as before and with (\ref{decLambda}) the operator $b_{\hat\Theta}:\Lambda^3\RR^{10}\to\Lambda^3\RR^{10}$ is given by
\begin{equation}
b_{\hat\Theta}=
\begin{pmatrix} 
& & 6\tilde e_\Theta\otimes * \\ 
& -3 b_\Theta \otimes * & \\
 e_\Theta\otimes * & & \end{pmatrix}\,.
\end{equation}
Its eigenvalues, eigenspaces and their dimensions are
\begin{equation*}
{\renewcommand{\arraystretch}{2}\begin{array}{|c|c|c|}\hline
0	
	& \Lambda^3_{(\frac{2}{3})}\RR^8
	& 48 \\\hline
\pm 18i
	& \Lambda^2_{(-6)}\RR^8\otimes\RR_\pm
	& 2\times 7\\\hline
\pm 6i 
	& \Lambda^2_{(2)}\RR^8\otimes \RR_\mp
	& 2\times 21 \\\hline
\pm 6\sqrt{7}i
	& \left\{\left(\mp \frac{i}{\sqrt{7}}\tilde e_\Theta(F),F\wedge \epsilon\right) \big| F\in\RR^8\right\}
	& 2\times 8 \\\hline
\end{array}}
\end{equation*}
\end{proposition}
Here we used the following Lemma similar to Lemmas \ref{lemmak=5} and \ref{lemmak=4}.
\begin{lemma}
The maps
\begin{equation}\begin{aligned}
e_\Theta &:\Lambda^3\RR^8\to\RR^8\,,\quad e_\Theta(F)_{l}=\Theta_{ijkl}F^{ijk}\\
\tilde e_\Theta &:\RR^8\to\Lambda^3\RR^8\,,\quad \tilde e_\Theta(F)_{jkl}=\Theta_{ijkl}F^i
\end{aligned}\end{equation}
obey 
\begin{equation}
e_\Theta\big|_{\Lambda^3_{(\frac{2}{3})}\RR^8}=0\,,\quad \tilde e_\Theta \circ e_\Theta\big|_{\Lambda^3_{(-4)}\RR^8}=-24 id\,,\quad e_\Theta \circ \tilde e_\Theta =-24 id\,.
\end{equation}
\end{lemma}

\section{An example with discrete symmetry}

On $V=\RR^8$ we consider the four form\footnote{We use the short notation $e_{ijkl}=e_i\wedge e_j\wedge e_k\wedge e_l$.} 
\begin{equation}
\Omega= e_{1234}+e_{2345}+e_{3456}+e_{4567}+e_{5678}+e_{6781}+e_{7812}+e_{8123}\,.
\end{equation}
This four-form is invariant under the action of $\mathbbm{Z}_8$ on $\Lambda^kV$ which is given by  $\sigma_a(e_{i_1\ldots i_k}):=e_{i_1+a\ldots i_k+a}$. We will denote the generator by $\sigma:=\sigma_1$.

$b_\Omega$ is defined on $\Lambda^2V$, $\Lambda^3V$ and $\Lambda^4V$. A careful calculation yields the following results.

{$\mathbf{[k=2]}$.\ }\ The minimal polynomial of $b_\Omega:\Lambda^2V\to\Lambda^2V$ is given by
\[
p(t)=t(t^2-1)(t^2-4)(t^2-2)(t^2-(1+\sqrt{2})^2)(t^2-(1-\sqrt{2})^2)
\] 
and the eigenvalues of $\sigma$ on $\Lambda^2V$ have multiplicities $3$ for $\pm1$ and $\pm i$, and $4$ for $\pm\frac{1}{\sqrt{2}}\pm \frac{i}{\sqrt{2}}$.

The eigenspaces $V_\beta$ for $\beta=0$, $\pm1$, $\pm2$, $\pm\sqrt{2}$, and $\pm1\pm\sqrt{2}$ as well as their behavior under $\sigma\in \mathbbm{Z}_8$ are explicitly given as follows.

\begin{equation}\label{k=2-1}\begin{aligned}
V_{\pm 1}  = {\rm span}\big\{
	v^1_\pm &= (e_{56}-e_{12})\pm(e_{38}+e_{47})\,,
	v^2_\pm = (e_{67}-e_{23})\pm(e_{58}-e_{14})\,,\\
	v^3_\pm &= (e_{78}-e_{34})\mp(e_{16}+e_{25})\,,
	v^4_\pm =-(e_{18}+e_{45})\mp(e_{27}+e_{36})
\big\}
\end{aligned}\end{equation}
with
$
v^1_\pm\overset{\sigma}{\longrightarrow}v^2_\pm\overset{\sigma}{\longrightarrow}v^3_\pm\overset{\sigma}{\longrightarrow}v^4_\pm\overset{\sigma}{\longrightarrow}-v^1_\pm
$
such that $\sigma^4+\mathbbm{1}$ is the minimal equation on $V_{\pm1}$. 

\begin{equation}
V_{\pm 2} = {\rm span}\left\{ v_\pm = e_{13}-e_{17}+e_{35}+e_{57}\mp( e_{24}-e_{28} +e_{46}+ e_{68})\right\}
\end{equation} 
with $\sigma(v_\pm)=\mp v_\pm$ such that $\sigma\pm\mathbbm{1}=0$ on $V_{\pm 2}$.

\begin{equation}\begin{aligned}
V_{\pm\sqrt{2}}  = {\rm span}\big\{ 
	v_1^\pm & = -e_{13}+e_{17}+e_{35}+e_{57}\mp\sqrt{2}(e_{28}+e_{46}), \\
	v_2^\pm & = -e_{24}-e_{28}-e_{46}+e_{68}\pm\sqrt{2}(e_{17}+e_{35})
\big\}
\end{aligned}\end{equation}
with $\sigma(v^\pm_1)=\mp\sqrt{2} v_1^\pm+v^\pm_2$ and $\sigma(v^\pm_2)=-v_1^\pm $, i.e.  $\sigma^2\pm\sqrt{2}\,\sigma+\mathbbm{1}=0$ is the minimal equation on $V_{\pm\sqrt{2}}$.

Moreover, for $\epsilon,\eta\in\{\pm1\}$ we have
\begin{equation}\begin{aligned}
V_{\epsilon +\eta \sqrt{2}}   = {\rm span}\big\{ 
	v^\eta_\epsilon & = e_{14}-\epsilon e_{27}+\epsilon e_{36}+e_{58}
								+(\epsilon+\eta\sqrt{2})(e_{23}-\epsilon e_{18}+\epsilon e_{45}+e_{67}),\\
	w^\eta_\epsilon & =\epsilon e_{25}-\epsilon e_{16}-e_{38}+e_{47}
								+(\epsilon+\eta\sqrt{2})(e_{12}+\epsilon e_{34}+e_{56}+\epsilon e_{78})
\big\}
\end{aligned}\end{equation}
with $v^\eta_\epsilon\overset{\sigma}{\longrightarrow}\epsilon w^\eta_\epsilon\overset{\sigma}{\longrightarrow}\epsilon v^\eta_\epsilon $ 
such that $\sigma^2-\epsilon\mathbbm{1}=0$ is the minimal equation on $V_{\epsilon +\eta \sqrt{2}}$.

Last but not least, 
\begin{equation}\begin{aligned}
V_0 = {\rm span}\big\{ & w_1= e_{24}+e_{28}-e_{46}+e_{68}, w_2=e_{13}+e_{17}-e_{35}+e_{57},\\
					   & e_{15},e_{26},e_{37},e_{48}\big\}	
\end{aligned}\end{equation}
with $e_{15}\overset{\sigma}{\longrightarrow}e_{26}\overset{\sigma}{\longrightarrow}e_{37}\overset{\sigma}{\longrightarrow}e_{48}\overset{\sigma}{\longrightarrow}-e_{15}$ and $w_1\overset{\sigma}{\longrightarrow}-w_2\overset{\sigma}{\longrightarrow}-w_1$. I.e.\ $\sigma^4+\mathbbm{1}=0$ and $\sigma^2+\mathbbm{1}=0$ are the minimal equations on $E={\rm span}\{e_{15},e_{26},e_{37},e_{48}\}$ and $W={\rm span}\{w_1,w_2\}$, respectively.

{$\mathbf{[k=3]}$.\ }\ On $\Lambda^3V$ the duality operator $3 b_\Omega$ has minimal polynomial 
\begin{equation}
p(t)= t(t^2-4)(t^2-2)(t^4-14t^2+16)
\end{equation}
such that the eigenvalues are given by $0$, $\pm2$, $\pm\sqrt{2}$, and $\pm\frac{\sqrt{22}}{2}\pm\frac{\sqrt{6}}{2}$.
Moreover, the eigenvalues of $\sigma$ have multiplicities $7$ each.

The respective eigenspaces and the action of $\sigma$ are given as follows.
\begin{equation}\begin{aligned}
V_{\pm2} = {\rm span}\big\{
	w_1^\pm & = e_{237}-e_{125}-e_{156}+e_{367}\pm (e_{138}-e_{134}+e_{457}-e_{578}),
\\
	w_2^\pm & = e_{348}-e_{236}-e_{267}+e_{478}\pm (e_{124}-e_{168}-e_{245}+e_{568}),
\\
	w_3^\pm & = e_{145}+e_{158}-e_{347}-e_{378}\pm (e_{167}-e_{127}+e_{235}-e_{356}),
\\
	w_4^\pm & = e_{126}-e_{148}+e_{256}-e_{458}\pm (e_{278}-e_{238}+e_{346}-e_{467}),
\\
	u_1^\pm & = e_{257}-e_{123}-e_{136}+e_{567}\pm (e_{158}-e_{145}+e_{347}-e_{378}),
\\
	u_2^\pm & = e_{368}-e_{234}-e_{247}+e_{678}\pm (e_{126}-e_{148}-e_{256}+e_{458} ) ,
\\
	u_3^\pm & = e_{147}+e_{178}-e_{345}-e_{358}\pm (e_{156}-e_{125}+e_{237}-e_{367}),
\\
	u_4^\pm & = e_{128}-e_{146}+e_{258}-e_{456}\pm (e_{267}-e_{236}+e_{348}-e_{478})
\big\}\,.
\end{aligned}\end{equation}
This basis is well adapted in the way that $w_1^\pm\overset{\sigma}{\longrightarrow}w_2^\pm\overset{\sigma}{\longrightarrow}w_3^\pm\overset{\sigma}{\longrightarrow}w_4^\pm\overset{\sigma}{\longrightarrow}w_1^\pm$ and 
$u_1^\pm\overset{\sigma}{\longrightarrow}u_2^\pm\overset{\sigma}{\longrightarrow}u_3^\pm\overset{\sigma}{\longrightarrow}u_4^\pm\overset{\sigma}{\longrightarrow}-u_1^\pm$, i.e. $\sigma^4-\mathbbm{1}=0$ and $\sigma^4+\mathbbm{1}=0$ are the respective minimal equations on $W^\pm={\rm span}\{w_i^\pm\}$ and $U^\pm={\rm span}\{u^\pm_i\}$.

For the zero eigenvalue we have
\begin{equation}\begin{aligned}
V_{0}=\  {\rm span}\big\{
	x_1 =\ & 	e_{236} -e_{267} +e_{348} -e_{478}\,,
	x_2 =  		e_{145} -e_{158} +e_{347} -e_{378}\,,\\	
	x_3 =\ & 	e_{256} -e_{126} -e_{148} +e_{458}\,,
	x_4 = 		e_{156} -e_{125} -e_{237} +e_{367}\,,\\
	y_1 =\ &	e_{123} -e_{136} +e_{257} -e_{567}\,,
	y_2 = 		e_{234} -e_{247} +e_{368} -e_{678}\,,\\
	y_3 =\ & 	e_{147} -e_{178} +e_{345} -e_{358}\,,
	y_4 = 		e_{258} -e_{128} -e_{146}  +e_{456}\,,\\
	u_1 =\ & 	e_{278} -e_{238} -e_{346} +e_{467}\,,
	u_2 =	 	e_{138} -e_{134} -e_{457} +e_{578}\,,\\
	u_3 =\ &	e_{124} +e_{168} -e_{245} -e_{568}\,,
	u_4 = 		e_{127} -e_{167} +e_{235} -e_{356}\,,\\
	v_1 =\ & 	e_{127}-e_{123}-e_{134}-e_{136}-e_{138}+e_{147}+e_{167}+e_{178} \\
		   & 	+e_{235}-e_{257}+e_{345}+e_{356}+e_{358}-e_{457}-e_{567}-e_{578}\,, \\
	v_2 =\ & 	e_{128}-e_{124}+e_{146}-e_{168}-e_{234}+e_{238}-e_{245}-e_{247} \\
		   & 	+e_{258}+e_{278}+e_{346}-e_{368}+e_{456}+e_{467}-e_{568}-e_{678}\,,\\
	w_1=\ & 	e_{127}-e_{123}+e_{134}-e_{136}+e_{138}-e_{147}+e_{167}-e_{178} \\
		    &	+e_{235}-e_{257}-e_{345}+e_{356}-e_{358}+e_{457}-e_{567}+e_{578} \\
	w_2=\ & 	e_{124}-e_{128}-e_{146}+e_{168}-e_{234}+e_{238}+e_{245}-e_{247} \\
			& 	-e_{258}+e_{278}+e_{346}-e_{368}-e_{456}+e_{467}+e_{568}-e_{678} \big\}\,.
\end{aligned}\end{equation}
The above basis obeys 
$x_1\overset{\sigma}{\longrightarrow}x_2\overset{\sigma}{\longrightarrow}x_3\overset{\sigma}{\longrightarrow}
x_4\overset{\sigma}{\longrightarrow}-x_1$, $y_1\overset{\sigma}{\longrightarrow}y_2\overset{\sigma}{\longrightarrow}y_3\overset{\sigma}{\longrightarrow}
y_4\overset{\sigma}{\longrightarrow}-y_1$ and $u_1\overset{\sigma}{\longrightarrow}u_2\overset{\sigma}{\longrightarrow}u_3\overset{\sigma}{\longrightarrow}
u_4\overset{\sigma}{\longrightarrow}-u_1$  
as well as 
$v_1 \overset{\sigma}{\longrightarrow}v_2 \overset{\sigma}{\longrightarrow}-v_1$ and 
$w_1\overset{\sigma}{\longrightarrow}w_2\overset{\sigma}{\longrightarrow}w_1$.
Therefore, the minimal equations are $\sigma^4+\mathbbm{1}=0$ on $X={\rm span}\{x_i\}$, $Y={\rm span}\{y_i\}$ and $U={\rm span}\{u_i\}$ as well as $\sigma^4-\mathbbm{1}=0$ on $V\oplus W$ for $V={\rm span}\{v_1,v_2\}$ and $W ={\rm span}\{w_1,w_2\}$ -- more precisely  $\sigma^2+\mathbbm{1}=0$ on  $V$ and $\sigma\mp\mathbbm{1}=0$ on $W^\pm={\rm span}\{w_1\pm w_2\}$.

Furthermore,
\begin{equation}\begin{aligned}
V_{\pm\sqrt{2}}={\rm span} \big\{
	v_1^\pm & = e_{168}-e_{124}-e_{245}+e_{568}\pm\sqrt{2}(e_{135}-e_{157}),\\
	v_2^\pm & = e_{127}+e_{167}-e_{235}-e_{356}\pm\sqrt{2}(e_{246}-e_{268}),\\
	v_3^\pm & = e_{238}+e_{278}-e_{346}-e_{467}\pm\sqrt{2}(e_{357}-e_{137}),\\
	v_4^\pm & = e_{134}+e_{138}-e_{457}-e_{578}\pm\sqrt{2}(e_{468}-e_{248})\big\}\,.
\end{aligned}\end{equation}
This basis is chosen in such a way that 
$v^\pm_1\overset{\sigma}{\longrightarrow}v^\pm_2\overset{\sigma}{\longrightarrow}v^\pm_3\overset{\sigma}{\longrightarrow}v^\pm_4\overset{\sigma}{\longrightarrow}-v^\pm_1$. 
Therefore, the minimal equation is $\sigma^4+\mathbbm{1}=0$ on both spaces.

Last but not least for $\beta\in\{\pm\frac{\sqrt{22}}{2}\pm\frac{\sqrt{6}}{2}\}$ we have
\begin{equation}\begin{aligned}
V_{\beta}={\rm span} \Big\{
		v^\beta_1 =\ &							
											(e_{126}+e_{148}+e_{256}+e_{458})
				+\frac{\beta}{4}			(e_{238}+e_{278}+e_{346}+e_{467}) \\
			&	+\frac{8-\beta^2}{4}	 	(e_{137}+e_{357})
				+\frac{2}{\beta}			(e_{234}+e_{678}) 
				+\frac{\beta^2-4}{2\beta} (e_{247}+e_{368})
				,\\
		v^\beta_2 =\ & 
											(e_{125}+e_{156}+e_{237}+e_{367})
				+\frac{\beta}{4}			(e_{134}+e_{138}+e_{457}+e_{578}) \\
			&	+\frac{8-\beta^2}{4} 		(e_{248}+e_{468})
				+\frac{2}{\beta}			(e_{178}+e_{345})  
				+\frac{\beta^2-4}{2\beta}	(e_{147}+e_{358})	 
				,\\
		v^\beta_3 =\ &
											(e_{236}+e_{267}+e_{348}+e_{478})
				+\frac{\beta}{4} 			(e_{124}+e_{168}+e_{245}+e_{568})\\
			&	+\frac{8-\beta^2}{4}		(e_{135}+e_{157})
				+\frac{2}{\beta}			(e_{128}+e_{456})
				+\frac{\beta^2-4}{2\beta} 	(e_{146}+e_{258})
				,\\
		v^\beta_4 =\ & 							
											(e_{145}+e_{158}+e_{347}+e_{378})
				+\frac{\beta}{4}			(e_{127}+e_{167}+e_{235}+e_{356})\\
			&	+\frac{8-\beta^2}{4}		(e_{246}+e_{268})
				+\frac{2}{\beta}			(e_{123}+e_{567})
				+\frac{\beta^2-4}{2\beta}	(e_{136}+e_{257})
				\Big\}\,.
\end{aligned}\end{equation}
This choice of basis obeys $v^\beta_1\overset{\sigma}{\longrightarrow}v^\beta_2\overset{\sigma}{\longrightarrow}v^\beta_3\overset{\sigma}{\longrightarrow}v^\beta_4\overset{\sigma}{\longrightarrow}-v^\beta_1$, such that $\sigma^4-\mathbbm{1}=0$ is the minimal equation for $\sigma$ on $V_\beta$. 

{$\mathbf{[k=4]}$.\ }\ On $\Lambda^4V$ the minimal polynomial of $6 b_\Omega$ is given by 
\begin{equation}
p(t)= t(t^2-4)(t^2-16)(t^2-8)
\end{equation}
and the eigenvalues $0$, $\pm2$, $\pm4$, and $\pm2\sqrt{2}$ have multiplicities $26$, $16$, $4$ and $2$, respectively. Moreover, the multiplicities of the eigenvalues of $\sigma$ are $10$ for $\pm i$, $9$ for $\pm 1$, and $8$ for $\pm\frac{1}{\sqrt{2}}\pm\frac{i}{\sqrt{2}}$. 
We will list here the low dimensional eigenspaces and we will show, how $\Omega$ is related to the eigenvalues $\pm 2\sqrt{2}$.

The eigenspaces to the eigenvalues $\pm4$ are given by
\begin{equation}\begin{aligned}
V_{\pm4}={\rm span} \Big\{
	v_1^{\pm}& = e_{1257}+e_{1356}+e_{2478}+e_{3468}\pm (e_{1347}-e_{1246}+e_{2568}-e_{3578}),\\
 	v_2^{\pm}& = e_{2368}+e_{2467}-e_{1358}-e_{1457}\mp (e_{1367}-e_{1468}+e_{2357}-e_{2458})  ,\\
 	w_1^{\pm}& = e_{1357}-e_{1458}-e_{2367}+e_{2468}\mp (e_{1368}+e_{1467}+e_{2358}+e_{2457}),\\
 	w_2^{\pm}& = e_{1256}-e_{1357}+e_{2468}-e_{3478}\pm (e_{1247}+e_{1346}-e_{2578}-e_{3568})\Big\}
\end{aligned}
\end{equation}
with 
$v^\pm_1\overset{\sigma}{\longrightarrow}v^\pm_2\overset{\sigma}{\longrightarrow}\mp v^\pm_1$
and 
$w^\pm_1\overset{\sigma}{\longrightarrow}-w^\pm_2\overset{\sigma}{\longrightarrow}-w^\pm_1$
such that the minimal equation of $\sigma$ is $\sigma^2+\mathbbm{1}=0$ on $V_{4}$, and $\sigma^4-\mathbbm{1}=0$ on $V_{-4}$.

The eigenspaces to the eigenvalues $\pm2\sqrt{2}$ are given by
\begin{equation}\begin{aligned}
V_{\pm 2\sqrt{2}}={\rm span} \Big\{
		u_1^{\pm}& = e_{2345}-e_{1238}-e_{1678}+e_{4567}\pm \sqrt{2}( e_{2367}-e_{1458} ),\\
	 	u_2^{\pm}& = e_{1234}+e_{1278}+e_{3456}+e_{5678}\pm \sqrt{2}( e_{1256}-e_{3478} )\Big\}
\end{aligned}
\end{equation}
with 
$ u^\pm_1\overset{\sigma}{\longrightarrow}u^\pm_2\overset{\sigma}{\longrightarrow}u^\pm_1$
such that $\sigma$ has eigenvalues $\pm 1$ on $V_{\pm2\sqrt{2}}$.

\begin{remark}
The two-dimensional $+1$-eigenspace of $\sigma$ within $V_{2\sqrt{2}}\oplus V_{-2\sqrt{2}}$ is given by ${\rm span}\big\{\Omega,\omega\big\}$ where 
\begin{equation}\begin{aligned}
\Omega&=\frac{1}{2}( u_1^+ + u_1^- + u_2^+ + u_2^- )\,,\\
\omega&:=\frac{1}{2}(u_1^+ - u_1^- +u_2^+ - u_2^-)\,.
\end{aligned}\end{equation}
These forms fulfill $b_\Omega(\Omega)=\frac{\sqrt{2}}{3}\omega$ and $b_\Omega(\omega)=\frac{\sqrt{2}}{3}\Omega$.
In particular, $\Omega$ itself is not an eigenform with respect to $b_\Omega$, in contrast to the discussion following Definition \ref{defperf}.
\end{remark}

We conclude this example by adding some comments on the eigenspaces of $b_\Omega$ to the remaining eigenvalues $0$ and $\pm2$ which we as usual denote by $V_0$ and $V_{\pm2}$. This explains the so far unusual asymmetry in the behavior of $\sigma$ on $V_{\pm4}$.

The map $\sigma$ acting $V_0$ has eigenvalues $\pm\frac{1}{\sqrt{2}}\pm\frac{i}{\sqrt{2}}$ with multiplicity $4$,  $\pm i$ with multiplicity $3$, as well as $\pm 1$ with multiplicity $2$. Restricted to $V_{\pm2}$ the eight  eigenvalues of $\sigma$ come with multiplicity $2$, each.

\section{Outlook}

The duality operator we defined here in flat space can be defined in the same way on a Riemannian or semi-Riemannian manifold. In particular, all that has been discussed for $\mathfrak{g}$-invariant duality operators can be transferred to manifolds with a $\mathfrak{g}$-structure. In this case the $\mathfrak{g}$-invariant differential form $\Omega\in\Omega^\ell(M)$ is parallel with respect to a connection associated to the given $\mathfrak{g}$-structure. 

One application of our duality relations may be the following. Let the manifold under consideration be spin, and take a connection on the spinor bundle $S$ on $M$. This connection and its curvature are locally described by elements in the exterior algebra of $M$, the so called $k$-form potentials and fluxes; see for example \cite{papadopoulos, klinker, klinker2}. The duality relation presented here may be a candidate to generalize the duality for metric connections on the base manifold $M$.


\begin{appendix}

\section{Useful Decompositions}\label{appendec}

We are interested in the decomposition of certain tensor products of irreducible representations of $\mathfrak{so}(n)$. We recall the decomposition of the tensor product of anti-symmetric powers of $V=\RR^n$ into irreducible $\mathfrak{gl}(n)$-modules. Let $k,\ell\leq \frac{n}{2}$ then 
\begin{equation}
\Lambda^\ell V\otimes\Lambda^k V
=\bigoplus_{i=0}^{\min\{k,\ell\}}  \llbracket k+\ell-i,i \rrbracket\,.
\end{equation}
Here  $\llbracket k+\ell-i,i \rrbracket$ denotes the irreducible representation space of weight $e_i+e_{k+\ell-i}$. 
With respect to $\mathfrak{so}(n)$ these spaces are reducible for $i\neq 0$. 
The irreducible components are obtained by  contraction with the metric. If we denote the trace free parts by  $\llbracket \cdot,\cdot\rrbracket_0$ we get 
\begin{equation}
\llbracket k+\ell-i,i\rrbracket =  \bigoplus_{j=0}^{i}\llbracket k+\ell-i-j,i-j\rrbracket_0
\end{equation}
which yields the final $\mathfrak{so}(n)$-decomposition
\begin{equation}\label{dec-so}
\Lambda^\ell V\otimes\Lambda^k V
=\bigoplus_{i=0}^{\min\{k,\ell\}}\bigoplus_{j=0}^{i}\llbracket k+\ell-i-j,i-j\rrbracket_0\,.
\end{equation}
Due to Hodge duality the preceding formula can be used for, say, $\ell>\frac{n}{2}$, too, we only have to insert $\Lambda^{n-\ell}V \approx \Lambda^\ell$ instead. For a more systematic treatment of such decompositions we refer the reader to the nice article \cite{KoikeTerada}.

By $\pi_m$ we denote the projection $\Lambda^kV\otimes\Lambda^\ell V\to\Lambda^mV=\llbracket m,0\rrbracket$.

We are in particular interested in the second symmetric power of $\Lambda^k V$. With the above notation for $k=\ell$ we have the following $\mathfrak{so}(n)$-decomposition
\begin{equation}\label{decompositionsquare}
\begin{split}
S^2(\Lambda^k V) =\ &
\bigoplus_{j=0}^{[\frac{k}{2}]}  \Big( \bigoplus_{i=0}^{k-2j}  \llbracket k+2j-i,k-2j-i \rrbracket_0 \Big) \\
=\ & \bigoplus_{j=0}^{[\frac{k}{2}]}  \Big( \bigoplus_{i=0}^{k-2j-1}  \llbracket k+2j-i,k-2j-i \rrbracket_0 \Big)
       \oplus \bigoplus_{j=0}^{[\frac{k}{2}]}  \Lambda^{4j}V \,.
\end{split}
\end{equation}
In particular $\Lambda^\ell V\subset S^2(\Lambda^k V)$ only if $\ell\equiv 0\mod4$.

\section{Some Calculations}\label{appencalc}

In this appendix we add the calculations for equations (\ref{b2}) to (\ref{b5}) that we left out in Lemma \ref{example-Theta} as well as the calculations for Lemma \ref{lemmak=5} and Proposition \ref{propositionk=5}.

\subsection{Calculations for Lemma \ref{example-Theta}}
We recall the content of Lemma \ref{example-Theta}: The duality map $b_\Theta:\Lambda^4\RR^8\to\Lambda^4\RR^8$ given by
$
b_\Theta(F)_{ijkl}=\Theta^{mn}{}_{[ij}F_{kl]mn}\,.
$
obeys
\begin{align*}
b_\Theta^2(F)_{ijkl} =\ & 
     	\tfrac{1}{6}\Theta^{mn}{}_{[ij}\Theta^{op}{}_{kl]}F_{mnop} 
		+\tfrac{2}{3}F_{ijkl} -\tfrac{8}{3}b_\Theta(F)_{ijkl} 
	\tag{\ref{b2}} \\
b_\Theta^3(F)_{ijkl} =\ &
		\tfrac{4}{3}F_{ijkl} + \tfrac{2}{3}b_\Theta(F)_{ijkl} -\tfrac{10}{3}b_\Theta^2(F)_{ijkl}
		+\tfrac{2}{9}\Theta_{[ijk}{}^{p}\Theta^{rsn}{}_{l]}F_{pnrs}   	
 	\tag{\ref{b3}}\\
b_\Theta^4(F)_{ijkl} =\ & 
		4 b_\Theta(F)_{ijkl} - \tfrac{8}{3}b^2_\Theta(F)_{ijkl} 
		-\tfrac{13}{3}b^3_\Theta(F)_{ijkl}+\tfrac{1}{9}\Theta_{ijkl}\Theta^{prsn}F_{prsn}
	\tag{\ref{b4}}\\
b_\Theta^5(F)_{ijkl}=\ & 
		-\tfrac{25}{3}b_\Theta^4(F)_{ijkl}-20b_\Theta^3(F)_{ijkl}
		-\tfrac{20}{3}b_\Theta^2(F)_{ijkl}+16b_\Theta(F)_{ijkl}
	\tag{\ref{b5}}
\end{align*}
To get (\ref{b2}) we calculate
\begin{align*}
b_\Theta^2(F)_{ijkl} 
=\  & \Theta^{mn}{}_{[ij}b_\Theta(F)_{kl]mn} \\
=\  & \Theta^{mn}{}_{[ij}\delta_{kl]mn}^{abcd}\Theta^{op}{}_{ab}F_{cdop}\\
=\  & \tfrac{1}{6}\Theta^{mn}{}_{[ij}\delta_{kl]}^{ab} \delta_{mn}^{cd}\Theta^{op}{}_{ab}F_{cdop}
     +\tfrac{1}{6}\Theta^{mn}{}_{[ij}\delta_{kl]}^{cd} \delta_{mn}^{ab}\Theta^{op}{}_{ab}F_{cdop}\\
	&-\tfrac{2}{3}\Theta^{mn}{}_{[ij}\delta_{kl]}^{ac} \delta_{mn}^{bd}\Theta^{op}{}_{ab}F_{cdop}\\
=\  & \tfrac{1}{6}\Theta^{mn}{}_{[ij}\Theta^{op}{}_{kl]}F_{mnop}
     +\tfrac{1}{6}\Big(12\delta_{[ij}^{op} - 4\Theta^{op}{}_{[ij} \Big) F_{kl]op} \\
    &  +\tfrac{2}{3}\delta^{i'j'k'l'}_{ijkl}\Big(6\delta_{n i'j'}^{opq}
		 -9\Theta_{[ni'}{}^{[op}\delta_{j']}^{q]}\Big)g_{k'q}F_{l'}{}^{n}{}_{op}\\
=\  & \tfrac{1}{6}\Theta^{mn}{}_{[ij}\Theta^{op}{}_{kl]}F_{mnop} +2F_{ijkl}-\tfrac{2}{3}b_\Theta(F)_{ijkl}\\
	&+4\delta^{i'j'k'l'}_{ijkl}\Big(\tfrac{1}{3}\delta^q_n\delta_{i'j'}^{op} 
	 +\tfrac{2}{3}\delta^o_n\delta_{i'j'}^{pq} \Big)g_{k'q}F_{l'}{}^{n}{}_{op} \\
	&-6\delta^{i'j'k'l'}_{ijkl}\Big(\tfrac{1}{3}\cdot\tfrac{1}{3}\Theta_{i'j'}{}^{op}\delta_{n}^{q}
	 + \tfrac{1}{3}\cdot\tfrac{2}{3}\Theta_{ni'}{}^{op}\delta_{j'}^{q} 
	 +\tfrac{2}{3}\cdot\tfrac{1}{3}\Theta_{i'j'}{}^{qo}\delta_{n}^{p}	\\
    &+\tfrac{2}{3}\cdot\tfrac{2}{3}\Theta_{ni'}{}^{qo}\delta_{j'}^{p} \Big)g_{k'q}F_{l'}{}^{n}{}_{op} \\
=\  & \tfrac{1}{6}\Theta^{mn}{}_{[ij}\Theta^{op}{}_{kl]}F_{mnop} +\tfrac{2}{3}F_{ijkl}
	 -\tfrac{2}{3}b_\Theta(F)_{ijkl}   \\
    &+\tfrac{2}{3}\Theta_{[ij}{}^{op} F_{kl]}{}_{op} -\tfrac{8}{3}\Theta_{mn [ij}F_{kl]}{}^{mn} \\
=\  & \tfrac{1}{6}\Theta^{mn}{}_{[ij}\Theta^{op}{}_{kl]}F_{mnop} 
	 +\tfrac{2}{3}F_{ijkl}-\tfrac{8}{3}b_\Theta(F)_{ijkl} \,.
\end{align*}

To get (\ref{b3}) we need the image of the first summand in (\ref{b2}) under $b_\Theta$.

\begin{align*}
\Theta^{mn}{}_{[ij}\delta^{abcd}_{kl]mn}&\Theta^{pq}{}_{ab}\Theta^{rs}{}_{cd}F_{pqrs}  \\
=\ & \tfrac{1}{6}  \Theta^{mn}{}_{[ij}\delta^{ab}_{kl]}\delta^{cd}_{mn} 
		\Theta^{pq}{}_{ab}\Theta^{rs}{}_{cd}F_{pqrs}
     +\tfrac{1}{6}\Theta^{mn}{}_{[ij}\delta^{cd}_{kl]}\delta^{ab}_{mn} \Theta^{pq}{}_{ab}
		\Theta^{rs}{}_{cd}F_{pqrs} \\
	&-\tfrac{2}{3}\Theta^{mn}{}_{[ij}\delta^{ac}_{kl]}\delta^{bd}_{mn} 
		\Theta^{pq}{}_{ab}\Theta^{rs}{}_{cd}F_{pqrs} \\
=\  & \tfrac{1}{3}\Big(12\delta^{rs}_{[ij} - 4\Theta^{rs}{}_{[ij}\Big)  \Theta^{pq}{}_{kl]}F_{pqrs}
      -\tfrac{2}{3}\delta_{ijkl}^{i'j'k'l'}\Theta_{mni'j'} \Theta^{mpqt}\Theta^{rsn}{}_{l'}g_{k't}F_{pqrs} \\
=\  & 4 \Theta_{pq[ij}F_{kl]}{}^{pq} - \tfrac{4}{3}\Theta^{rs}{}_{[ij}  \Theta^{pq}{}_{kl]}F_{pqrs} \\
    &-\tfrac{2}{3}\delta_{ijkl}^{i'j'k'l'}\Big(6\delta_{ni'j'}^{pqt} 
	 -9\Theta_{[ni'}{}^{[pq}\delta^{t]}_{j']}\Big)\Theta^{rsn}{}_{l'}g_{k't}F_{pqrs} \\
=\  & 4 b_\Theta(F)_{ijkl} - \tfrac{4}{3}\Theta^{rs}{}_{[ij} \Theta^{pq}{}_{kl]}F_{pqrs} 
     -\tfrac{4}{3}\Theta^{rs}{}_{[ij}F_{kl]rs}\\
    &+ \tfrac{2}{3}\delta_{ijkl}^{i'j'k'l'}\Theta_{i'j'}{}^{pq}\delta^{t}_{n}
			\Theta^{rsn}{}_{l'}g_{k't}F_{pqrs}
     + \tfrac{4}{3}\delta_{ijkl}^{i'j'k'l'}\Theta_{ni'}{}^{pq}\delta^{t}_{j'}
		\Theta^{rsn}{}_{l'}g_{k't}F_{pqrs}\\
	&+\tfrac{4}{3} \delta_{ijkl}^{i'j'k'l'}\Theta_{i'j'}{}^{tp}\delta^{q}_{n}
		\Theta^{rsn}{}_{l'}g_{k't}F_{pqrs}
     +\tfrac{8}{3} \delta_{ijkl}^{i'j'k'l'}\Theta_{ni'}{}^{tp}\delta^{q}_{j'}
		\Theta^{rsn}{}_{l'}g_{k't}F_{pqrs} \\
=\  & \tfrac{8}{3} b_\Theta(F)_{ijkl} - \tfrac{4}{3}\Theta^{rs}{}_{[ij} \Theta^{pq}{}_{kl]}F_{pqrs} 
     + \tfrac{2}{3}\Theta^{pq}{}_{[ij}\Theta^{rs}{}_{kl]}F_{pqrs} \\
	&+\tfrac{4}{3}\Theta_{[ijk}{}^{p}\Theta^{rsn}{}_{l]}F_{pnrs}
     +\tfrac{8}{3} \delta_{ijkl}^{i'j'k'l'}\Theta_{npi'k'}\Theta^{nrst}g_{l't}F^{p}{}_{j'rs} \\
=\ 	& \tfrac{8}{3} b_\Theta(F)_{ijkl} - \tfrac{2}{3}\Theta^{rs}{}_{[ij} \Theta^{pq}{}_{kl]}F_{pqrs}  
	 +\tfrac{4}{3}\Theta_{[ijk}{}^{p}\Theta^{rsn}{}_{l]}F_{pnrs} \\
    &+\tfrac{8}{3} \delta_{ijkl}^{i'j'k'l'}
		\Big(6\delta_{pi'k'}^{rst}-9\Theta_{[pi'}{}^{[rs}\delta^{t]}_{k']}\Big)g_{l't}F^{p}{}_{j'rs} \\
=\  & \tfrac{8}{3} b_\Theta(F)_{ijkl} - \tfrac{2}{3}\Theta^{rs}{}_{[ij} \Theta^{pq}{}_{kl]}F_{pqrs}  
	 +\tfrac{4}{3}\Theta_{[ijk}{}^{p}\Theta^{rsn}{}_{l]}F_{pnrs} \\
    &+\tfrac{16}{3}\delta_{ijkl}^{i'j'k'l'} \delta_{i'k'}^{rs}\delta^{t}_{p} g_{l't}F^{p}{}_{j'rs} 
	 +\tfrac{32}{3} \delta_{ijkl}^{i'j'k'l'} \delta_{i'k'}^{tr}\delta_{p}^{s} g_{l't}F^{p}{}_{j'rs} \\
    &-\tfrac{24}{3}\cdot\tfrac{1}{3} \delta_{ijkl}^{i'j'k'l'}\Theta_{i'k'}{}^{rs}
		\delta^{t}_{p} g_{l't}F^{p}{}_{j'rs}
     -\tfrac{24}{3}\cdot\tfrac{2}{3}\delta_{ijkl}^{i'j'k'l'}\Theta_{pi'}{}^{rs}
		\delta^{t}_{k'}g_{l't}F^{p}{}_{j'rs} \\
	&-\tfrac{48}{3}\cdot\tfrac{1}{3} \delta_{ijkl}^{i'j'k'l'} \Theta_{i'k'}{}^{tr}
		\delta^{s}_{p} g_{l't}F^{p}{}_{j'rs}
	 -\tfrac{48}{3}\cdot\tfrac{2}{3}\delta_{ijkl}^{i'j'k'l'}\Theta_{pi'}{}^{tr}
		\delta^{s}_{k'} g_{l't}F^{p}{}_{j'rs}\\
=\  & \tfrac{8}{3} b_\Theta(F)_{ijkl} - \tfrac{2}{3}\Theta^{rs}{}_{[ij} \Theta^{pq}{}_{kl]}F_{pqrs}  
	 +\tfrac{4}{3}\Theta_{[ijk}{}^{p}\Theta^{rsn}{}_{l]}F_{pnrs} \\
	&+\tfrac{16}{3}F_{ijkl}  
	 -\tfrac{8}{3} b_\Theta(F)_{ijkl}
	 +\tfrac{32}{3}b_\Theta(F)_{ijkl}\\
=\  & \tfrac{16}{3}F_{ijkl} + \tfrac{32}{3} b_\Theta(F)_{ijkl} 
	 -\tfrac{2}{3}\Theta^{rs}{}_{[ij} \Theta^{pq}{}_{kl]}F_{pqrs}  
	 +\tfrac{4}{3}\Theta_{[ijk}{}^{p}\Theta^{rsn}{}_{l]}F_{pnrs}\,.
\end{align*}

So we get for the third power of $b_\Theta$

\begin{align*}
b_\Theta^3(F)_{ijkl}
=\  & b_\Theta(b_\Theta^2(F))_{ijkl} \\
=\  &\tfrac{2}{3}b_\Theta(F)_{ijkl}-\tfrac{8}{3}b_\Theta^2(F)_{ijkl} \\
	&+\tfrac{1}{6}\Big( 
	  \tfrac{16}{3}F_{ijkl} + \tfrac{32}{3} b_\Theta(F)_{ijkl} 
	  -\tfrac{2}{3}\Theta^{rs}{}_{[ij} \Theta^{pq}{}_{kl]}F_{pqrs}  \\
	&+\tfrac{4}{3}\Theta_{[ijk}{}^{p}\Theta^{rsn}{}_{l]}F_{pnrs} \Big)\\
=\ &\tfrac{2}{3}b_\Theta(F)_{ijkl}-\tfrac{8}{3}b_\Theta^2(F)_{ijkl} \\
   &+\tfrac{8}{9}F_{ijkl} 
	+\tfrac{16}{9}b_\Theta(F)_{ijkl} 
    - \tfrac{2}{3}\Big(b_\Theta^2(F)_{ijkl} -\tfrac{2}{3}F_{ijkl} +\tfrac{8}{3}b_\Theta(F)_{ijkl}\Big)\\
   &+\tfrac{2}{9}\Theta_{[ijk}{}^{p}\Theta^{rsn}{}_{l]}F_{pnrs}  \\
=\ & \tfrac{4}{3}F_{ijkl} + \tfrac{2}{3}b_\Theta(F)_{ijkl}-\tfrac{10}{3}b_\Theta^2(F)_{ijkl}
      +\tfrac{2}{9}\Theta_{[ijk}{}^{p}\Theta^{rsn}{}_{l]}F_{pnrs} \,.
\end{align*}

To evaluate $b^4_\Theta$, i.e.\ (\ref{b4}),  we need the image of $\Theta_{[ijk}{}^{p}\Theta^{rsn}{}_{l]}F_{pnrs}$ under $b_\Theta$,

\begin{align*}
 \Theta^{mn}{}_{[ij}\delta_{kl]mn}^{abcd}&\Theta_{abc}{}^{o}\Theta^{prs}{}_{d}F_{oprs} \\
=\  & \tfrac{1}{4}\Theta^{mn}{}_{[ij} \delta_{kl]m}^{abc}\delta_n^d\Theta_{abc}{}^{o}
		\Theta^{prs}{}_{d}F_{oprs} 
	 -\tfrac{3}{4}\Theta^{mn}{}_{[ij} \delta_{kl]m}^{abd}\delta_n^c\Theta_{abc}{}^{o}
		\Theta^{prs}{}_{d}F_{oprs} \\
=\  & \tfrac{1}{4}\Theta^{mn}{}_{[ij}\Theta_{kl]m}{}^{o}\Theta^{prs}{}_{n}F_{oprs} \\
	&-\tfrac{3}{4}\Theta^{mn}{}_{[ij} \Big(\tfrac{1}{3}\delta_{kl]}^{ab}\delta_{m}^{d}
	 +\tfrac{2}{3} \delta_{kl]}^{da}\delta_{m}^{b} \Big)\delta_n^c\Theta_{abc}{}^{o}
		\Theta^{prs}{}_{d}F_{oprs} \\
=\  & \tfrac{1}{2}\Theta^{mn}{}_{[ij}\Theta_{kl]m}{}^{o}\Theta^{prs}{}_{n}F_{oprs} 
     +\tfrac{1}{2}\Theta_{mn[ij} \Theta_{k}{}^{mno}\Theta^{prs}{}_{l]}F_{oprs} \\
=\  &\tfrac{1}{2}\delta_{ijkl}^{i'j'k'l'}\Big( 6\delta_{ni'j'}^{tuo} 
		-9 \Theta_{[ni'}{}^{[tu}\delta_{j']}^{o]}\Big)\Theta^{prsn}g_{tk'}g_{ul'}F_{oprs} \\
	&+\tfrac{1}{2}\delta_{ijkl}^{i'j'k'l'}\Big( 12\delta_{i'j'}^{to} -4\Theta_{i'j'}{}^{to}\Big)
		\Theta^{prs}{}_{l'}g_{k't}F_{oprs} \\
=\ 	&-2 \Theta_{[ijk}{}^{o}\Theta^{prs}{}_{l]}F_{oprs} 
	 -\delta_{ijkl}^{i'j'k'l'}\Big(\tfrac{1}{2}\Theta_{i'j'k'l'}\Theta^{prsn}F_{nprs} \\
	&+\Theta_{ni'k'l'}\Theta^{prsn}F_{j'prs}
         - \Theta_{i'j'k'}{}^{o}\Theta^{prs}{}_{l'}F_{oprs}\Big)\\
=\ 	&- \Theta_{[ijk}{}^{o}\Theta^{prs}{}_{l]}F_{oprs}  -  \tfrac{1}{2}\Theta_{ijkl}\Theta^{prsn}F_{nprs} \\
	&+ \delta^{i'j'k'l'}_{ijkl}\Big( 6\delta_{i'j'k'}^{prs} 
		-9\Theta_{[i'j'}{}^{[pr}\delta_{k']}^{s]}\Big) F_{l'prs}\\
=\ 	&- \Theta_{[ijk}{}^{o}\Theta^{prs}{}_{l]}F_{oprs}  +  \tfrac{1}{2}\Theta_{ijkl}\Theta^{prsn}F_{prsn} 
	 -6 F_{ijkl} +  9 b_\Theta(F)_{ijkl}\,.
\end{align*}

This is 

\begin{align*}
b^4_\Theta(F)_{ijkl}=\ & b_\Theta(b^3_\Theta(F))_{ijkl} \\
=\ 	&\tfrac{4}{3}b_\Theta(F)_{ijkl} + \tfrac{2}{3}b^2_\Theta(F)_{ijkl}-\tfrac{10}{3}b^3_\Theta(F)_{ijkl} \\
    &+\tfrac{2}{9}\Big(\!-\Theta_{[ijk}{}^{o}\Theta^{prs}{}_{l]}F_{oprs}  
	 +  \tfrac{1}{2}\Theta_{ijkl}\Theta^{prsn}F_{prsn} 
	 -6 F_{ijkl} +  9 b_\Theta(F)_{ijkl}\Big)\\
=\  &-\tfrac{4}{3}F_{ijkl}+\tfrac{10}{3}b_\Theta(F)_{ijkl} 
	 + \tfrac{2}{3}b^2_\Theta(F)_{ijkl} -\tfrac{10}{3}b^3_\Theta(F)_{ijkl}\\
	&-\Big( b_\Theta^3(F)_{ijkl} -\tfrac{4}{3}F_{ijkl} - \tfrac{2}{3}b_\Theta(F)_{ijkl}
	 +\tfrac{10}{3}b_\Theta^2(F)_{ijkl} \Big) \\
	&+\tfrac{1}{9}\Theta_{ijkl}\Theta^{prsn}F_{prsn} \\
=\  &4 b_\Theta(F)_{ijkl} - \tfrac{8}{3}b^2_\Theta(F)_{ijkl} -\tfrac{13}{3}b^3_\Theta(F)_{ijkl}
	 +\tfrac{1}{9}\Theta_{ijkl}\Theta^{prsn}F_{prsn}\,.
\end{align*}

The last step is easy. For $b_\Theta^5$ we need the image of $\Theta_{ijkl}\Theta^{prsn}F_{prsn}$. This is a multiple of $\Theta_{ijkl}$ for which we have $b_\Theta (\Theta)_{ijkl}= \Theta_{mn[ij}\Theta^{mn}{}_{kl]} = -4\Theta_{ijkl}$. This yields
\begin{align*}
b^5_\Theta(F)_{ijkl}
=\ & b_\Theta(b^4_\Theta(F))_{ijkl}\\
=\ & 4b_\Theta^2(F)_{ijkl}-\tfrac{8}{3}b_\Theta^3(F)_{ijkl}-\tfrac{13}{3}b_\Theta^4(F)_{ijkl}
	 -\tfrac{4}{9}\Theta_{ijkl}\Theta^{opqr}F_{opqr}\\
=\ & -\tfrac{25}{3}b_\Theta^4(F)_{ijkl}-20b_\Theta^3(F)_{ijkl}
	 -\tfrac{20}{3}b_\Theta^2(F)_{ijkl}+16b_\Theta(F)_{ijkl}\,.
\end{align*}

\subsection{Calculations for Lemma \ref{lemmak=5} and Proposition \ref{propositionk=5}}

The proof of Lem\-ma \ref{lemmak=5} is a straightforward calculation. 
The maps
\begin{equation*}\tag{\ref{formk=5}}\begin{aligned}
d_\Theta & :\Lambda^5\RR^8\to\Lambda^3\RR^8,\quad d_\Theta(F)_{lmn}=\Theta_{ijk[l}F^{ijk}{}_{mn}\,,
	\\
\tilde d_\Theta & : \Lambda^3\RR^8\to\Lambda^5\RR^8,\quad \tilde d_\Theta(F)_{jklmn}=\Theta_{i[jkl}F^i{}_{mn]}\,.
\end{aligned}\end{equation*}
are isomorphisms and connected to $b_\Theta$ and to the Hodge operator via
\begin{equation*}\tag{\ref{calck=5}}
d_\Theta\circ\tilde d_\Theta  	= -\frac{6}{5}id+\frac{3}{2}b_\Theta(F)\,,
\quad\text{and}\quad
*\, d_\Theta\, * 				= -20 \tilde d_\Theta \,.
\end{equation*}
A consequence of this is $\tilde d_\Theta \circ d_\Theta =- *\, d_\Theta\circ \tilde d_\Theta , *$.

We make use of (\ref{traces}) and get
\begin{align*}
d_\Theta\tilde d_\Theta(F)_{lmn}
=\ & \Theta_{ijk[l}\tilde d_\Theta(F)^{ijk}{}_{mn]}\\
=\ & \delta_{lmn}^{abc}\Theta^{ijk}{}_{a}\Theta_{s[ijk}F^s{}_{bc]}\\
=\ & \delta_{lmn}^{abc}\big( \tfrac{1}{10}\Theta^{ijk}{}_{a}\Theta_{sijk}F^s{}_{bc}
		+\tfrac{3}{10}\Theta^{ijk}{}_{a}\Theta_{sbci}F^s{}_{jk} \\
	&	-\tfrac{6}{10}\Theta^{ijk}{}_{a}\Theta_{sbij}F^a{}_{kc}\big)\\
=\ & -\tfrac{21}{5}F_{lmn} 
		-\tfrac{3}{5}\delta_{lmn}^{abc}\big( -6\delta_{m'k}\delta_{sl'}-4\Theta_{sklm} \big)F^{sk}{}_{n'}\\
   &    -\tfrac{3}{10}\delta_{lmn}^{abc}\big(2\delta^{jk}_{bc}\delta_{sa} 
		-4\Theta_{a}{}^{j}{}_{sb}\delta^{k}_{c} - \Theta^{jk}{}_{bc}\delta_{sa}\big)F^s{}_{jk}\\
=\ & -\tfrac{21}{5}F_{lmn} -\tfrac{3}{5}F_{lmn} -\tfrac{6}{5}\Theta_{jk[lm}F^{jk}{}_{n]}
		+\tfrac{3}{10}\Theta_{jk[lm}F^{jk}{}_{n]} \\
	&   +\tfrac{18}{5}F_{lmn}+\tfrac{12}{5}\Theta_{jk[lm}F^{jk}{}_{n]}\\
=\ & -\tfrac{6}{5}F_{lmn} +	\tfrac{3}{2} b_\Theta(F)_{lmn}	\,.
\end{align*}
Furthermore we have
\begin{align*}
*d_\Theta(*F)_{ijklm}
=\ & -\tfrac{1}{6}\epsilon_{ijklmnop}d_\Theta(*F)^{nop} 
=\  -\tfrac{1}{6}\epsilon_{ijklmnop}\Theta_{abc}{}^n(*F)^{abcop}\\
=\ & \frac{1}{36}\epsilon_{ijklmnop}\epsilon^{abcoprst}\Theta_{abc}{}^nF_{rst}
=\  40 \delta^{abcrst}_{ijklmn}\Theta_{abc}{}^nF_{rst}\\
=\ & 20 \delta^{abcrs}_{ijklm}\delta^{t}_{n}\Theta_{abc}{}^nF_{rst} 
=\  -20 \Theta_{n[ijk}F^{n}{}_{lm]} \\
=\ &  -20 \tilde d_\Theta(F)_{ijklm} 
\end{align*}
The proof of Proposition \ref{propositionk=5} is divided into three cases $b_{\hat\Theta}(\hat F)$ where we consider $\hat F= F^5$, $\hat F=F^4\wedge V$, and $\hat F=F^3\wedge \epsilon$ with $F^k\in\Lambda^k\RR^8$ and $V\in\RR^2$, separately. 

First, we consider $\hat F=F$ and get
\begin{align*}
b_{\hat\Theta}(F)_{lmnop}
=\ & \hat\Theta_{ijk[lmn} F^{ijk}{}_{op]} \\
=\ & 15 \Theta_{[ijka}\epsilon_{bc]} F^{ijk}{}_{de}\delta^{abcde}_{lmnop}  
=\  3 \Theta_{ijk[l}\epsilon_{mn}F^{ijk}{}_{op]} \\
=\ & 3 b_\Theta(F)_{[lmn}\epsilon_{op]} 
=\  \tfrac{3}{10}(d_\Theta(F)\wedge\epsilon)_{lmnop} \\
=\ & \tfrac{3}{10} (d_\Theta\otimes *)(F)_{lmnop}
\end{align*}
Second, we insert $\hat F=F\wedge \epsilon$ which yields
\begin{align*}
b_{\hat\Theta}(F\wedge \epsilon)_{lmnop}
=\ & \hat\Theta_{ijk[lmn}(F\wedge \epsilon)^{ijk}{}_{op]} \\
=\ &15 \Theta_{[ijka}\epsilon_{bc]}(F\wedge \epsilon)^{ijk}{}_{de}\delta^{abcde}_{lmnop} \\
=\ &30\big( 	\Theta^{ijk}{}_{a}\epsilon_{bc}
				+ 3\Theta_{ab}{}^{ij}\epsilon^{k}{}_{c}
				-\Theta_{abc}{}^{i}\epsilon^{jk}
				\big) F_{[ijk}\epsilon_{de]}\delta^{abcde}_{lmnop} \\
=\ &  	3\Theta_{ijk[l}\epsilon_{mn}F^{ijk}\epsilon_{op]}
				+ 54 \Theta_{ij[lm}F^{[ij}{}_{n}\epsilon^{k]}{}_{o}\epsilon_{p]k} 
				-  9 \Theta_{[lmn}{}^{i}\epsilon^{jk}F_{op][i}\epsilon_{jk]}  \\
=\ & -3 \Theta_{[lmn}{}^{i}\epsilon^{jk}F_{op]i}\epsilon_{jk} 
=\  6\tilde d_\Theta(F)_{lmnop} \\
=\ & 6(\tilde d_\Theta\otimes *)(F\wedge \epsilon)_{lmnop} \\
\end{align*}
Last but not least, for $\hat F=F\wedge V$ we get
\begin{align*}
b_{\hat\Theta}(F\wedge V)_{lmnop}
=\ & \hat\Theta_{ijk[lmn}(F\wedge V)^{ijk}{}_{op]} \\
=\ &15 \Theta_{[ijka}\epsilon_{bc]}(F\wedge V)^{ijk}{}_{de}\delta^{abcde}_{lmnop} \\
=\ &15 \big( 
			\Theta^{ijk}{}_{a}\epsilon_{bc}
			+3\Theta_{ab}{}^{ij}\epsilon^k{}_{c}
			-\Theta_{abc}{}^{i}\epsilon^{jk} \big)F_{[ijkd}V_{e]}\delta^{abcde}_{lmnop} \\
=\ &\tfrac{3}{2} \big( 
			4\Theta^{ijk}{}_{a}\epsilon_{bc}F_{[ijk]d}V_{e}
			+18\Theta_{ab}{}^{ij}\epsilon^k{}_{c}F_{de[ij}V_{k]}
			\big)\delta^{abcde}_{lmnop} \\		
=\ & 9 \Theta_{ij[lm}{}F^{ij}{}_{no}\epsilon^k{}_{p]}V_{k}
=\ 9 b_\Theta(F)_{[lmno}(*V)_{p]} \\
=\ &\tfrac{9}{5}(b_\Theta\otimes *)(F\wedge V)_{lmnop}
\end{align*}
The result on the eigenspaces and eigenvalues may be checked by applying $b_{\hat\Theta}$ and using Lemma \ref{calck=5}.

\end{appendix}


\begin{thebibliography}{AB}

\bibitem{AlekCorDev1}
D.V. Alekseevsky, V.~Cortes, and C.~Devchand.
\newblock Partially-flat gauge fields on manifolds of dimension greater than four.
\newblock {\tt hep-th/0205030}, 2002.

\bibitem{AlekCorDev2}
D.V. Alekseevsky, V.~Cortes, and C.~Devchand.
\newblock {Y}ang-{M}ills connections over manifolds with {G}rassmann structure.
\newblock {\em J.~Math.~Phys.} 44 (2003), no.~12, 6047-6076.

\bibitem{AtHiSi}
M.~F. Atiyah, N.~J. Hitchin, and I.~M. Singer.
\newblock Self-duality in four-dimensional {R}iemannian geometry.
\newblock {\em Proc.~Roy.~Soc.~London Ser.~A} 362 (1978), no.~1711, 425-461.

\bibitem{BauKanSin}
Laurent Baulieu, Hiroaki Kanno, and I.~M. Singer.
\newblock Special quantum field theories in eight and other dimensions.
\newblock {\em Commun.~Math.~Phys.} 194 (1998), no.~1, 149-175.

\bibitem{Baulieu}
Laurent Baulieu and C\'{e}line Laroche
\newblock On generalized self-duality equations towards supersymmetric quantum field theories of forms.
\newblock {\em Modern Phys.~Lett. A} 13 (1998), no. 14, 1115–1132. 

\bibitem{bryant}
Robert L.~Bryant.
\newblock Metrics with exeptional holonomy.
\newblock {Ann.~Math.} 126 (1987), no.~3, 525-576.

\bibitem{CGK}
E.~Corrigan, P.~Goddard, and A.~Kent.
\newblock Some comemnts on the ADHM construction in $4k$ dimensions.
\newblock {\em Commun.~Math.~Phys.} 100 (1985), no.~1, 1-13.

\bibitem{CDF}
E.~Corrigan, C.~Devchand, D.~B.~Fairlie, and J.~Nuyts.
\newblock First-order equations for gauge fields in spaces of dimension greater than four.
\newblock {\em Nucl.~Phys.~B} 214 (1983), 452-464.

\bibitem{DevDual}
Y.~Brihaye, C.~Devchand, and J.~Nuyts.
\newblock Self-duality for eight-dimensional gauge theories.
\newblock {\em Phys.~Rev.~D (3)} 32 (1985), no.~4, 990-994.

\bibitem{DNW}
C.~Devchand, J.~Nuyts, and G.~Weingart.
\newblock Matryoshka of Special Democratic Forms.
\newblock {\em Comm.~Math.~Phys.} 293 (2010), no.~2, 545-562.

\bibitem{Fernandez}
M.~Fernandez.
\newblock A classification of {R}iemannian manifolds with structure group $Spin(7)$.
\newblock {\em Ann.~Mat.~Pura Appl.~(4)} 143 (1986), no.~1, 101-122.

\bibitem{vector}
R.~B.~Brown and A.~Gray.
\newblock Vector cross products.
\newblock {\em Comment.~Math.~Helv.} 42 (1967), 222-236.

\bibitem{Gauntlett2003}
Jerome P.~Gauntlett, Jan B.~Gutowski, and Stathis Pakis.
\newblock The geometry of $D=11$ null {K}illing spinors.
\newblock{\em J. High Energy Phys.} 2003, no.~12, 049, 29 pp.

\bibitem{papadopoulos}
U.~Gran, J.~Gutowski, G.~Papadopoulos, and D.~Roest.
\newblock Aspects of spinorial geometry.
\newblock {\em Modern Phys. Lett. A} 22 (2007), no. 1, 1–16. 

\bibitem{klinker}
F.~Klinker.
\newblock The torsion of spinor connection and related structures.
\newblock {\em SIGMA Symmetry Integrability Geom. Methods Appl.} 2 (2006), Paper 077, 28 pp. 

\bibitem{klinker2}
F.~Klinker.
\newblock SUSY structures on deformed supermanifolds.
\newblock {\em Differential Geom. Appl.} 26 (2008), 566-582.

\bibitem{KoikeTerada}
Kazuhiko Koike and Itaru Terada.
\newblock Young-Diagrammatic Methods for the Representation Theory of the Classical Groups of type $B_n$, $C_n$, $D_n$.
\newblock {\em J. Algebra} 107 (1987), no.~2, 466-511.

\bibitem{Tsimpis2006}
Dimitros Tsimpis.
\newblock M-theory on eight-manifolds revisited: $\mathcal{N}=1$ supersymmetry and generalized $Spin(7)$ structures.
\newblock{\em  J. High Energy Phys.} 2006, no.~4, 027, 26 pp.

\bibitem{Ward}
R.~S.~Ward.
\newblock Completely solvable gauge-field equations in dimensions greater than four.
\newblock{\em Nucl.~Phys.~} B236 (1984), 381-396.

\end{thebibliography}

\end{document}